\documentclass[10pt,a4paper]{article}
\usepackage[a4paper,margin=2.5cm]{geometry}
\usepackage{authblk}
\usepackage[utf8]{inputenc}
\usepackage[english]{babel}
\usepackage{amsmath}
\usepackage{amsfonts}
\usepackage{amssymb}
\usepackage{graphicx}
\usepackage{subcaption}
\usepackage{fancyhdr}
\usepackage{enumitem}
\usepackage{listings}
\usepackage{booktabs}
\usepackage{color}
\usepackage{moresize}
\usepackage{caption}
\usepackage{epstopdf}
\usepackage{soul}
\usepackage{bm}
\usepackage{cancel}
\newtheorem{theorem}{Theorem}[section]

\newtheorem{corollary}[theorem]{Corollary}

\newtheorem{example}[theorem]{Example}

\newtheorem{lemma}[theorem]{Lemma}

\newtheorem{proposition}[theorem]{Proposition}
\newtheorem{remark}[theorem]{Remark}

\newenvironment{proof}[1][Proof]{\noindent\textbf{#1.} }{\ \rule{0.5em}{0.5em}}
\def\C{{\mathcal C}}
\def\PP{{\mathbb P}}
\def\RR{{\mathbb R}}
\def\NN{{\mathbb N}}

\newcommand{\vp}{{V_{n}^{m}}}
\newcommand{\fil}{{\mu_{n,j}^{m}}}

\usepackage[colorlinks=true]{hyperref}
\definecolor{myblue}{rgb}{0.21, 0.34, 0.74}
\definecolor{mygreen}{rgb}{0, 0.32, 0}
\definecolor{myred}{rgb}{0.79, 0.0, 0.09}

\usepackage{orcidlink}

\begin{document}

\title{De la Vall\'ee Poussin type approximation \\ for solving  some Fredholm integral equations}

\author[1]{Domenico Mezzanotte \orcidlink{0000-0001-5154-6538}}
\author[1,2]{Donatella Occorsio \orcidlink{0000-0001-9446-4452}}
\author[2,3]{Mario Pezzella \orcidlink{0000-0002-1869-945X}}
\author[2]{Woula Themistoclakis \orcidlink{0000-0002-6185-1154}}

\affil[1]{\small Department of Basic and Applied Sciences, University of Basilicata, Potenza, Italy}
\affil[2]{\small Institute for Applied Mathematics "Mauro Picone", National Research Council of Italy, Naples, Italy}
\affil[3]{\small Department of Mathematics and Applications "Renato Caccioppoli", University of Naples Federico II, Naples, Italy}

\date{}

\maketitle

\begin{abstract}
In the present paper, we introduce a numerical method for second-kind Fredholm integral equations (FIEs) based on de la Vall\'ee Poussin-type (VP) polynomial approximations at Jacobi zeros. This class of approximations offers several advantages over classical Lagrange interpolation at the same nodes. In particular, it guarantees uniformly bounded Lebesgue constants in suitable weighted function spaces and provides near-best uniform approximation for functions in these spaces, while also significantly mitigating the Gibbs phenomenon. We show how these properties can be exploited in the numerical solution of FIEs. In particular, the proposed approach effectively handles functions with possible algebraic endpoint singularities and kernel functions featuring weak singularities or highly oscillatory behavior. Under suitable assumptions, we prove stability and convergence of the method in weighted uniform spaces. Furthermore, we develop an efficient implementation based on the solution of a well-conditioned linear system. Numerical results confirm the theoretical error estimates and show that the proposed method achieves higher local accuracy than the corresponding Lagrange-based projection method.
\end{abstract}


\section{Introduction}\label{sec:intro}
The paper deals with the numerical solution of the following type of second kind Fredholm Integral Equation (FIE) 
\begin{equation}\label{fie}
	f(y)-\nu \int_{-1}^{1} f(x)k(x,y)w(x)\ dx=g(y), \quad y\in (-1,1),
    \ \  \ \nu\in\mathbb{R}\setminus\{0\}, 
\end{equation}
where $w(x)=v^{\rho,\sigma}(x):=(1-x)^{\rho}(1+x)^{\sigma},$ with $\rho,\sigma >-1,$ is a given Jacobi weight, $g(y)$ and $k(x,y)$ are known functions defined on $(-1,1)$ and $(-1,1)^2,$ respectively, and $f$ is the unknown function to be determined. The equation \eqref{fie} can be written in operator form as follows
\begin{equation}\label{eq:FIE_Compact}
    (I-\nu K)f=g,
\end{equation}
where $I$ denotes the identity operator and  
\begin{equation}\label{eq:Integral_Operator}
     Kf(y)=\int_{-1}^{1} f(x)k(x,y)w(x) \ dx, \qquad y\in(-1,1).
\end{equation}
Among the extensive literature devoted to the numerical solution of FIEs (see, e.g., \cite{atkinson,kress2014,JMN}), we focus on the case in which the known functions may exhibit endpoint singularities that can be controlled by a Jacobi weight $u,$ generally different from $w$. This setting naturally favors to study the equation in the space $C_u$ of locally continuous functions endowed with weighted uniform norms. A global approach, hereafter referred to as the $\mathcal{L}$-method, based on the Lagrange projector $\mathcal{L}_n(w)$ associated with the weight $w$ in \eqref{fie}, has been proposed in \cite{dbmsiam}. The resulting projection method requires the solution of a linear system of $n$ equations. However, due to the unboundedness of the Lebesgue constants of $\mathcal{L}_n(w)$, some restrictions on the choice of weight $u$ and a higher regularity of $g\in C_u$ are needed to ensure its convergence and stability. 

The aim of the present study is is to overcome the above limitations by employing the \emph{de la Vall\'ee Poussin} (VP) polynomial quasi-projector $V_n^m(v^{\alpha,\beta})$, associated with a suitable Jacobi weight $v^{\alpha,\beta}$, which does not need to coincide with $w$. This operator, originally introduced in the trigonometric approximation \cite{VP1919}, has been first generalized to the algebraic setting in \cite{Themistoclakis_1999} and then studied in several papers \cite{Occothemi,OT_Dolomites,themi2011,themiL1} and applied to solve several singular integral equations \cite{De_Bonis_2021,Airfoil,ORT_2024}. Similarly to classical Lagrange interpolation, it is computed from evaluations of the target function at the Jacobi zeros of order $n$ corresponding to $v^{\alpha,\beta}.$ However, it interpolates the function at these nodes only in some special cases (e.g., $|\alpha|=|\beta|=1/2$). Its main peculiarity is the dependence on an additional parameter $1<m<n$. The uniform boundedness of the Lebesgue constants and the uniform convergence in $C_u$ with the same approximation order of the best  polynomial approximation in $C_u$, for suitable choices of $\alpha$ and $\beta,$ are assured when $m\sim n\sim (n-m),$ as $n\to\infty.$ Moreover, for a given number $n$ of nodes, we can suitably modulate $m$ to improve the local approximation and reduce the Gibbs phenomenon arising, for instance,  in presence of steep function variations.

We assume that, for all $g\in C_u$, the equation \eqref{eq:FIE_Compact} admits a unique solution $f^*\in C_u.$  To approximate this solution, we consider the finite-dimensional equation
\begin{equation}\label{eq_finita}
    (I-\nu K_n^m)f_n^m=g_n^m,
\end{equation}
where $f_n^m$ is the unknown polynomial,  obtained by replacing $Kf$ and $g$ in \eqref{eq:FIE_Compact}  with their respective VP approximations, namely 
\begin{equation*}
    K_n^mf:=V_n^m(v^{\alpha,\beta}, Kf), \qquad \text{and} \qquad g_n^m:=V_n^m(v^{\alpha,\beta}, g).
\end{equation*}
Although $V_n^m(v^{\alpha,\beta})$ is generally not a polynomial projector, we prove that equation \eqref{eq_finita} admits a unique polynomial solution $f_n^m$ that can be determined  by solving a linear system of $n$ equations. The proposed method,  hereafter referred to as the $VP$-method, inherits the advantages of VP approximation over Lagrange interpolation. These advantages emerge from both the theoretical and computational perspectives.

In particular, we prove that $f_n^m$ converges to $f^*$ in $C_u$ with a near-best approximation order. The same result can be achieved by the $\mathcal{L}$-method only in H\"older--Zygmund type subspaces of $C_u$, with an additional $\log n$ factor in the error estimates. Moreover, contrarily to the $\mathcal{L}$-method, our analysis covers the more general setting in which the weight $v^{\alpha,\beta}$ does not necessarily coincide with the weight $w=v^{\rho,\sigma}$ defining the integral operator in \eqref{eq:Integral_Operator}. This possibility, enlarge significantly the range of applicability requiring no restriction on the  space $C_u$ where the equation is studied. Furthermore, when both $\mathcal{L}$-method and $VP$-method can be applied and converge, we point out that the additional $\log n$ factor does not produce a significant difference in the maximum error produced by the two methods. Nevertheless, significant gains can be observed for the pointwise error, especially in presence of Gibbs-type oscillations. Finally, we remark that, similarly to the $\mathcal{L}$-method, the construction of the linear system associated with \eqref{eq_finita} requires the computation of the modified moments $Kp_n$ with $p_n$ an arbitrary Jacobi polynomial. We provide a detailed analysis of this problem in the special case of highly oscillatory kernels and weakly singular kernels exhibiting algebraic or logarithmic singularities. In the latter case, as a by-product,  we establish sufficient conditions ensuring the compactness of the operator $K$ in arbitrary weighted spaces $C_u$ and in H\"older--Zygmund subspaces of $C_u$. To our knowledge, these findings are known in literature only for particular weights $u$ and $w$.

The paper is structured as follows. Section~\ref{sec:preliminari} contains preliminary notation and a collection of  well-established results supporting the main findings of  Section~\ref{sec:mainresult}. There,  the numerical method is constructed and analyzed from a theoretical viewpoint.  Specifically, stability and convergence are established, and the associated finite-dimensional linear system is derived. A thorough comparison with the Lagrange projection method of \cite{dbmsiam} is carried out in Section~\ref{sec:LagrColl}. Section~\ref{sec:moments} is devoted to the computation of  recurrence relations for modified moments associated with weakly singular and highly oscillating kernels. Several numerical experiments are presented in Section~\ref{sec:numtest}.  Finally, Section~\ref{sec:proofs}  details the proofs of the main results, whereas Section~\ref{sec:concl} is devoted to the conclusions and a discussion of the obtained results.

\section{Preliminary results}\label{sec:preliminari}
Throughout the paper, we denote by $\C$ any positive constant having different meanings in various occurrences, and the notation $\C \neq \C(a,b,\dots)$ is used to express that $\C$ does not depend on $a,b,\dots$. Moreover, if $A, B >0$ are quantities depending on some parameters, the writing $A \sim B,$ means that there exists a constant $\C \neq \C(A, B)$  such that $\C^{-1} B \leq A  \leq \C B$.

As usual, by $\PP_m$ we denote the space of the algebraic polynomials of degree at most $m$ and the notation $k_x(y)$ (or $k_y(x)$) is used for a bivariate function $k(x,y)$ as a function of the only variable $y$ (or $x$).

\vspace{0.5cm}

\subsection{Function spaces}
Let $u(x)=v^{\gamma,\delta}(x):=(1-x)^\gamma(1+x)^\delta$ with $\gamma,\delta \geq 0$ and let $C_u$ be the space of the locally continuous functions $f$ on $(-1,1)$ satisfying the limit conditions
\begin{equation}\label{limCu}
	\lim\limits_{x\to 1^-}f(x)u(x)=0, \quad \text{if} \ \gamma>0 \qquad \text{and} \qquad
	\lim\limits_{x\to -1^+}f(x)u(x)=0, \quad \text{if} \ \delta>0.
\end{equation}
In the case $\gamma=\delta=0$, the space $C_u$ coincides with the space $C([-1,1])$ of continuous functions in $[-1,1]$.

Equipped with the norm
$$\|f\|_{C_u}:=\|fu\|_\infty= \sup_{x\in [-1,1]}|(fu)(x)|,$$
$C_u$ is a Banach space.
For any operator $A:C_u\to C_u$, let 
$\|A\|=\sup_{f\ne 0}\frac{\|Af\|_{C_u}}{\|f\|_{C_u}}$ denote its operator norm.

Denoting by
\begin{equation*}
E_m(f)_u:=\inf_{P\in\PP_m}\|f-P\|_{C_u},
\end{equation*}
the error of best polynomial approximation of $f\in C_u$, the limit conditions \eqref{limCu} are necessary to ensure that (see for instance \cite{mastromilobook})
\begin{equation*}
 \lim_{m\rightarrow \infty} E_m(f)_u=0,\quad \forall f\in C_u.
\end{equation*}

For any $f\in C_u$, the main part of the \textit{$\varphi$-modulus of smoothness} \cite[p. 90]{DT} of order $k\in \NN$ is defined as 
\[
\Omega_\varphi^k(f,t)_u=\sup_{0<h\leq t} \|\Delta^k_{h\varphi} f u\|_{I_{kh}},
\]
where $I_{kh}=[-1+(2kh)^2,1-(2kh)^2]$, and
\begin{gather*}
\Delta^k_{h\varphi} f(x)= \sum_{j=0}^k (-1)^j\left(\begin{array}{c}k \\j\end{array}\right) f\left(x+(k-2j)\frac{h}2 \varphi(x)\right).
\end{gather*}
Through $\Omega_\varphi^k(f,t)_u$ it is possible to define the H\"older-Zygmund space  $Z_r(u)$ of order $r\in \RR^+$  as  
\[Z_r(u)=\left\{f\in C_{u} \ : \ \sup_{t> 0}\frac{\Omega_{\varphi}^k(f,t)_{u}}{t^r} <\infty, \quad k>r \right\},\]
equipped with the  norm
\begin{equation*}
	\|f\|_{Z_r(u)}=\|f\|_{C_u} +  \sup_{t> 0}\frac{\Omega_{\varphi}^k(f,t)_{u}}{t^r}.
\end{equation*}
In the case $\gamma=\delta=0$, i.e. $u = 1$, we set $Z_r:=Z_r(u).$ 
Finally, we recall that for any $f \in Z_r(u)$ and $m$ sufficiently large, (say $m>m_0$)
\begin{equation}\label{stimaBAE}
	E_m(f)_u  \leq \mathcal{C} \, \frac{\lVert f \rVert_{Z_r(u)}}{m^r}, \quad \C \neq \C(f,m).
\end{equation}
Finally, $L^1$ will denote the space of summable functions defined as 
$$ L^1:=L^1([-1,1])=\left\{f:\ \|f\|_{1}=\int_{-1}^1|f(x)|dx<\infty.\right\}$$ 

\subsection{The solvability of the integral equation}

Equation \eqref{fie} has been widely studied in the literature and several theorems exist regarding the existence of a unique solution and its smoothness properties (see e.g. \cite{atkinson,kress2014,JMN} and the references there in).

For our aims, a crucial role is played by the compactness of the maps $K:C_u\to C_u$ and $K:Z_r(u)\to Z_r(u)$. These assumptions  and  $\ker(I-\nu K)=\{0\}$ allow us to apply the Fredholm alternative theorem, ensuring that for any function $g\in C_u$ (resp. $g\in Z_r(u)$), equation \eqref{fie} admits  a unique solution $f^*\in C_u$ (resp. $f\in Z_r(u)$).

There are several sufficient conditions in the literature on the kernel $k(x,y)$ that guarantee the compactness of the operator $K$ (see for instance, \cite{dbmsiam,JMN,mastromilobook}). In the present study, we focus on two particular cases: 
\begin{enumerate}
    \item 
      Highly oscillating kernels 
      $$K^\omega f (y)=\int_{-1}^1
      k^\omega(x,y) f(x) w(x) dx, \quad \text{with} \qquad k^\omega(x,y)=h(y) \begin{cases} 
      \sin(\omega x), \\
      \cos(\omega x), \end{cases}\ \ \omega\in \RR. $$
    \item Weakly singular kernels
    $$K^\mu f(y)=\int_{-1}^1 f(x)k^\mu(x,y)w(x)dx, \quad \text{with} \quad k^\mu(x,y)=\begin{cases}
    |x-y|^\mu, & \mu>-1, \ \mu \neq 0,\\
    \log|x-y|, & \mu=0,
\end{cases}$$\vspace{-0.5cm} 
\end{enumerate}
where $y\in(-1,1)$ and $w=v^{\rho,\sigma}$.

Regarding the solvability of equation \eqref{fie} in the case of highly oscillatory kernels, the following theorem follows from a result in \cite{dbmsiam}. 
 \begin{theorem}\label{alternativa_fredholm}
    If $\frac w u\in L^1$ 
then $K^\omega : C_u\to C_u$ is a compact operator. Moreover,  for any $r>0$,  the map $K^\omega: C_u\to Z_r(u)$ is bounded, and $K^\omega: Z_r(u)\to Z_r(u)$ is compact.    
\end{theorem}
 Regarding the case of weakly singular kernels, 
   we point out that in the case $w=u\equiv 1,$ the compactness of $K^\mu:C([-1,1])\to C([-1,1])$  is well-known (see, e.g.,
 \cite[p.11]{atkinson}. Moreover, the compactness of $K^\mu$ between Zygmund spaces with different weights, has been proved in the context of Cauchy Singular Integral Equations (see, e.g., \cite[Lemma 4.1]{iwota}).
In the general case here considered, we state the following results, that, to our knowledge, are not present in literature.

\begin{lemma}\label{boundedness}
For all $\mu>-1$, the operator $K^\mu:C_u\to C_u$   is compact  provided that one of the following assumptions is satisfied
\begin{eqnarray*}
\mu\ge 0 & \mbox{and\ } &\quad \frac{w(x)}{u(x)}\in L^1,\\  
\\ -1<\mu<0 & \mbox{and\ } &\quad  \frac{w(x)(1-x^2)^{\mu}}{u(x)}\in L^1.\end{eqnarray*}
\end{lemma}

\begin{lemma}\label{compattezza_zyg}
Let $\mu>-1.$ If $\frac w u\in C([-1,1])$ then  $K^\mu:C_u\to Z_r(u)$ is a bounded operator for any  $r\le 1+\mu$ if $\mu\ne 0$ and for all $0<r<1$ if $\mu=0$. Moreover, for the same values of $r$, $K^\mu:Z_r(u)\to Z_r(u)$ is compact. \end{lemma}

\subsection{VP approximation}
Given a Jacobi weight $v^{\alpha,\beta}(x), \, \alpha,\beta>-1$, let us denote by $p_n(v^{\alpha,\beta},x)$ the Jacobi orthonormal polynomial of degree $n\in\NN_0=\{0,1,\ldots\}$ with positive leading coefficient. Moreover, let us denote by  $\{x_k:=x_{n,k}(v^{\alpha,\beta})\}_{k=1}^n$ the zeros of $p_n(v^{\alpha,\beta},x)$ and by $\{\lambda_{k}:=\lambda_{n,k}(v^{\alpha,\beta})\}_{k=1}^n$  the Christoffel numbers related to $v^{\alpha,\beta},$ namely
\begin{equation}\label{lambda}
    \lambda_{k}:=\lambda_{n,k}(v^{\alpha,\beta})= \left(\sum_{i=0}^{n-1} p_i^2(v^{\alpha,\beta},x_{k})\right)^{-1}, \qquad k=1,\ldots, n.
\end{equation} 
For any pair of integers $1<m<n$, setting
\begin{equation*}
	\fil :=
	\begin{cases}
		1 & 0\leq j \leq n-m,\\
		\dfrac{n+m-j}{2m} & n-m < j < n+m,\\
		0 & j \geq n+m,
	\end{cases}
\end{equation*}	
the fundamental VP polynomials associated with $v^{\alpha,\beta}$ are defined by \cite{themi2011}
\begin{equation}\label{fi}
    \Phi_{n,k}^m(x) := \lambda_{k} \sum_{j=0}^{n+m-1} \fil p_j(v^{\alpha,\beta},x_{k})p_j(v^{\alpha,\beta},x), \qquad  \qquad |x|\leq 1, 
\end{equation}
and the VP polynomial approximating an arbitrary function $f$ defined on $[-1,1]$, is given by
\begin{equation}\label{VP-sum}
    \vp(v^{\alpha,\beta},f,x):= \sum_{k=1}^n f(x_{k}) \Phi_{n,k}^m(x), \qquad \qquad |x|\leq 1.
\end{equation}
The properties of these polynomials, the uniform boundedness of the operator $V_n^m(v^{\alpha,\beta}):f\to V_n^m(v^{\alpha,\beta},f)$ and the error estimates of the VP approximation have been studied in the literature for several functional spaces (see, e.g., \cite{themi2011,themiBarel,themiL1,Occothemi,capothemi,mata}). In what follows, we recall some results relevant to our purposes.

First of all, we highlight the similarity of the definitions \eqref{fi}--\eqref{VP-sum} with the following classical Lagrange polynomial $\mathcal{L}_n(v^{\alpha,\beta},f)\in \PP_{n-1}$ interpolating $f$ at the same zeros $\{x_k\}_{k=1}^n$ 
\begin{equation}\label{eq:Lagrange}
    \mathcal{L}_n(v^{\alpha,\beta},f,x)=\sum_{k=1}^n f(x_k)\ell_{n,k}(x),\quad \text{where} \quad \ell_{n,k}(x)=\lambda_k\sum_{j=0}^{n-1}p_j(v^{\alpha,\beta},x_k)p_j(v^{\alpha,\beta},x).
\end{equation}
It is well--known that the Lebesgue constant of Lagrange interpolation (namely the norm of  $\mathcal{L}_n(v^{\alpha,\beta}): C_u\to C_u$) increases with $n$ at least as $\log n$ (see, e.g. \cite{mastromilobook}). Furthermore, the bound
\begin{equation*}
    \max_{1\le i\le n} |f(x_i)|u(x_i)\le \|\mathcal{L}_n(v^{\alpha,\beta},f)\|_{C_u}\le \C \log n \left[\max_{1\le i\le n} |f(x_i)|u(x_i)\right],\quad \C\ne \C(n,f), \quad \quad \ \forall f\in C_u,
\end{equation*}
holds if and only if 		
\begin{equation}	\label{assumlag}
   \max\left(0,\ \frac\alpha 2+\frac 14\right)\le \gamma\le\frac\alpha 2+\frac 54,\qquad \max\left(0,\ \frac\beta 2+\frac 14\right)\le \delta\le\frac\beta 2+\frac 54.
\end{equation}
On the contrary, the norm of the VP operator $V_n^m(v^{\alpha,\beta}):C_u\to C_u$ can be uniformly bounded as stated in the following.
\begin{theorem}\label{th-VPinf}
\cite[Theorem 4.1]{themiBarel}
Let $\theta\in (0,1)$ be arbitrarily fixed, $n\in\NN$ be arbitrarily large, and $m=\lfloor \theta n\rfloor$.  If the Jacobi weights $v^{\alpha,\beta}$ and $u=v^{\gamma,\delta}$, with $\gamma,\delta\ge 0$, satisfy the following conditions
\begin{eqnarray}
\label{hp1-inf}
    && -1<\gamma-\delta-\frac{\alpha-\beta}2<1,\\
    \label{hp2-inf}
    && \max\left(0,\ \frac\alpha 2-\frac 14\right)<\gamma\le\frac\alpha 2+\frac 54,\\
    \label{hp3-inf}
    && \max\left(0,\ \frac\beta 2-\frac 14\right)<\delta\le\frac\beta 2+\frac 54,
\end{eqnarray}
then the map $V_n^m(v^{\alpha,\beta}):C_u\to C_u$ is uniformly bounded with respect to $n$, and
\begin{equation}\label{errVP-inf}
    \lim_{n\to\infty} \|V_n^m(v^{\alpha,\beta},f)-f\|_{C_u}=0, \qquad  \forall f\in C_u,
\end{equation}
holds at the same convergence rate of $E_n(f)_u$
\end{theorem}
It is also well--known that $\mathcal{L}_n(v^{\alpha,\beta})$ is a projection on $\PP_{n-1}$, whereas  the VP operator satisfies
\begin{equation}\label{inva}
    V_n^m(v^{\alpha,\beta}, P)=P, \qquad \forall P\in \PP_{n-m}.
\end{equation}
Moreover, the polynomial range space of $\vp (v^{\alpha,\beta}),$ given by
\begin{equation}\label{Snm}
    S_{n}^m := \text{span} \left\{ \Phi_{n,i}^m(x) \ : \ i=1,\ldots, n\right\},
\end{equation}
is nested between two canonical polynomial spaces 
\begin{equation*}
    \PP_{n-m} \subseteq S_n^m \subset \PP_{n+m-1},
\end{equation*}
and, as proved with the following result, has dimension $n$.
\begin{theorem}\label{prop-dual}
    For any Jacobi weight $v^{\alpha,\beta}$ and any pair of integers $1<m<n$, the fundamental VP polynomials $\{\Phi_{n,k}^m(x)\}_{k=1}^n$ defined in \eqref{fi}  are linearly independent.   
\end{theorem}
A very special case is when $v^{\alpha,\beta}$ is a Chebyshev  weight. In such a case, the following theorem establishes that the VP operator is a projector on $S_n^m$ satisfying  the same interpolation property of the Lagrange polynomial~\eqref{eq:Lagrange}. 
\begin{theorem}\cite{themi2011}\label{prop-Cheb}
    If $|\alpha|=|\beta|=1/2$, then for all $n,m\in\NN$ with $m<n$ we have that  
    \begin{equation}\label{interp}
        \Phi_{n,i}^m(x_j)=\delta_{i,j},\qquad i,j=1,\dots,n,
    \end{equation}
    and hence, for any function $f$ defined on $ [-1,1],$ we get
    \begin{equation}\label{interp-VP}
        V_n^m(v^{\alpha,\beta},f, x_j)=f(x_j),\qquad j=1,\dots,n.\qquad 
    \end{equation}
    Moreover,  $\vp$ is a projector on $S_n^m$ satisfying
    \begin{equation}\label{proj-VP}
        f\in S_n^m \Longleftrightarrow   f=V_n^m(v^{\alpha,\beta}, f).
    \end{equation}
\end{theorem}
Furthermore, regarding the  uniform boundedness of the Lebesgue constants of VP operator, we have the following result which improves Theorem \ref{th-VPinf}.
\begin{theorem}\cite{Occothemi}\label{th-Marci}
 Let $\theta\in (0,1)$ be arbitrarily fixed, $n\in \NN$ arbitrarily large and $m=\lfloor \theta n\rfloor$. For $u=v^{\gamma,\delta}$ with $\gamma,\delta\ge 0$  and $|\alpha|=|\beta|=1/2$, we have
\begin{equation}\label{Marci}
    \|V_n^n(v^{\alpha,\beta}, f)\|_{C_u} \sim \max_{1\le i\le n}|f(x_i)|u(x_i),   
\end{equation}
if and only if
\begin{itemize}
    \item Case $\alpha=\beta=-1/2$:\hspace{0.92cm} $0\le\gamma\le 1$ and $0\le\delta\le 1.$ 
    \item Case $\alpha=\beta=1/2$:\hspace{1.2 cm}  $0<\gamma\le 3/2$ and $0<\delta\le 3/2$ and $-1\le \gamma-\delta\le 1.$
    \item Case $\alpha= -1/2, \beta=1/2$:\hspace{0.3cm}  $0\le \gamma\le 1$ and $0<\delta\le 3/2$ and $\gamma-\delta\le 1/2.$
    \item Case $\alpha= 1/2, \beta=-1/2$: \hspace{0.15cm} $0< \gamma\le 3/2$ and $0\le \delta\le 1$ and $\gamma-\delta\ge - 1/2.$
\end{itemize}
\end{theorem}

In conclusion, we recall that VP polynomials generally  reduce the Gibbs phenomenon, and provide a better pointwise approximation. For instance, Figure \ref{fig:compL-VP} shows the absolute errors attained when approximating the function
\begin{equation*}
    f(x)=\frac{1}{1+1000\left(\frac 3 5-x\right)^2}+\frac{1}{1+1000\left(\frac 3 5+x\right)^2},
\end{equation*}
by the VP polynomial $V_n^m(v^{\alpha,\beta},f)$ and the Lagrange one $\mathcal{L}_n(v^{\alpha,\beta},f)$, for $n=80,$ $\theta=0.2$ and $\alpha=\beta=-\frac 1 2$. In such a case, \eqref{assumlag} and \eqref{hp1-inf}--\eqref{hp3-inf} are satisfied, and the maximum errors are almost comparable, but the VP polynomial achieves a more precise pointwise approximation. 
\begin{figure}[htbp]
    \centering
     \includegraphics[scale=0.4]{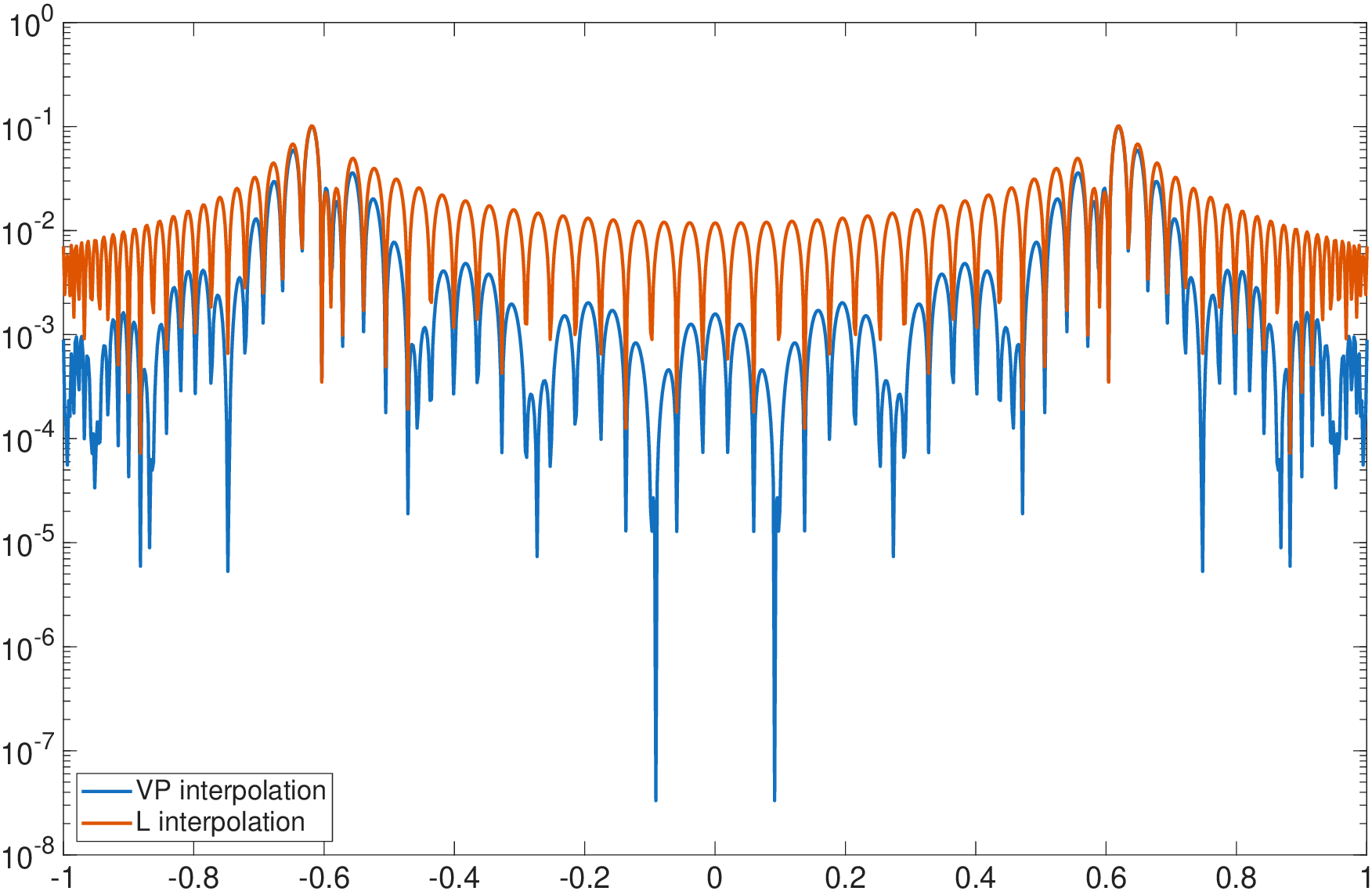} 
    \caption{Plot of the absolute errors of $f$ by $V_n^m\left(v^{-\frac 1 2,-\frac 1 2},f\right)$ (blue), and by $\mathcal{L}_n\left(v^{-\frac 1 2, -\frac 1 2},f\right)$ (red), for $n=80$ and $m=16$ ($\theta=0.2$).\hspace{3cm}}
    \label{fig:compL-VP}
\end{figure}

\section{The \texorpdfstring{$VP$}{VP}-method for FIE}\label{sec:mainresult}
In this section, we use the de la Vall\'e Poussin operator $V_n^m(v^{\alpha,\beta})$ to approximate the solution of equation \eqref{fie} in the weighted space $C_u$.  From now on, we will refer to the resulting method as the \emph{$VP$-method}. We underline that studying equation~\eqref{fie} in $C_u$ allows us to deal with the right-hand side $g$ with possible singularities in $\pm 1$  and the kernel $k$ presenting singularities at the endpoints with respect to the external variable $y$.

The $VP$-method is obtained by applying the quasi-projector $V_n^m(v^{\alpha,\beta})$ to \eqref{eq:FIE_Compact}, leading to the finite-dimensional equation
\begin{equation}\label{eqdiscr}
    (I-\nu K_n^m)f_n^m=g_n^m,
\end{equation}
where
\begin{equation*}
    K_n^mf(y)=\vp(v^{\alpha,\beta}, Kf,y)=\sum_{k=1}^nKf(x_k)\Phi_{n,k}^m(y),
\end{equation*}
\begin{equation*}   g_n^m(y)=\vp(v^{\alpha,\beta},g,y)=\sum_{k=1}^ng(x_k)\Phi_{n,k}^m(y).
\end{equation*}

The next theorem ensures that if \eqref{fie} has a unique solution $f^*\in C_u$ then also equation \eqref{eqdiscr} admits a unique solution that uniformly converges to $f^*$ under suitable assumptions. 
\begin{theorem}\label{convergence}
    Assuming $\mathrm{ker}(I-\nu K)=\{0\}$ in $C_u$, and $K: C_u\to C_u$  compact, let $f^*$ be the unique solution of \eqref{fie} for a given right-hand side $g\in C_u$.  
 Under the  hypotheses of Theorem \ref{th-VPinf} (Theorem \ref{th-Marci} in the case $|\alpha|=|\beta|=\frac 12$), for $n$ sufficiently large (say $n>n_0$), the finite dimensional equation \eqref{eqdiscr} admits a unique solution $f_n^{m} \in S_{n}^m$ satisfying the following estimate
    \begin{equation}\label{errore_equazione}
        \| (f^*-f_n^{m}) \|_{C_u} \leq \C \left[ E_{n-m}(g)_u +  E_{n-m}(Kf^*)_u \right], \quad \C \neq \C(n,f^*).
    \end{equation}
    Moreover, we have
    \begin{equation}\label{lim-cond}
     \lim_{n\to\infty}\|I-\nu K_n^m\|\|(I-\nu K_n^m)^{-1}\|=\|I-\nu K\|\|(I-\nu K)^{-1}\|.    
    \end{equation}
\end{theorem}

In the case that the solution $f^*$ belongs to some Zygmund subspace of $C_u$, the following corollary states that the approximate solution $f_n^m$ converges to $f^*$ at the same rate of $E_n(f^*)_u.$
    
\begin{corollary}\label{cor1}
Assume that $\mathrm{ker}(I-\nu K)=\{0\}$ and that $K:C_u\to Z_r(u)$ is bounded for some $r>0$. For any  $g\in Z_s(u)$ $s\le r$, equation \eqref{fie} has a unique solution $f^*\in Z_s(u)$ . Moreover, under the  hypotheses of Theorem \ref{th-VPinf} (Theorem \ref{th-Marci} in the case $|\alpha|=|\beta|=\frac 12$), the approximate equation  \eqref{eqdiscr} has a unique solution $f_n^m$ which satisfies the following estimate 
\begin{equation}\label{eqcor1}
        \| (f^*-f_n^{m})\|_{C_u} \le  \C  \ \frac{\|g \|_{Z_s(u)}}{n^s},  \quad \C \neq \C(n,f^*).
    \end{equation}
\end{corollary}

Note that the weight $w$ is fixed by the integral operator $K$ in~\eqref{eq:Integral_Operator} and the weight $u$ characterizes the space to which the solution belongs. On the other hand, the weight
$v^{\alpha,\beta}$ that defines~\eqref{eqdiscr} can be chosen, according to~\eqref{hp1-inf}-\eqref{hp3-inf}, to improve the approximation. We refer the reader to Section~\ref{sec:LagrColl} for further details.

The computation of the resulting approximate solution $f_n^m$ is investigated in the next subsection.    

\subsection{The final linear system}

Since  the  approximate solution of \eqref{fie} belongs to  $S_n^m$, we can represent it in  the basis $\mathcal{B}=\left\{\frac{\Phi_{n,j}^m(y)}{u(x_j)}\right\}_{j=1}^n$ of $S_n^m$, as follows
\begin{equation*}
    f_n^m(y)=\sum_{j=1}^n d_j \frac{\Phi_{n,j}^m(y)}{u(x_j)},\qquad |y|\le 1,
\end{equation*}
where the coefficients $d_j$ are unknown. The expansion of the discrete equation~\eqref{eqdiscr} in the same basis reads 
\begin{equation*}
    \sum_{i=1}^n d_i\frac{\Phi_{n,i}^m(y)}{u(x_i)} -\nu \sum_{i=1}^n \left[ \sum_{j=1}^n  \frac{u(x_i)}{u(x_j)} d_j\int_{-1}^1 k(x,x_i)\Phi_{n,j}^m(x) w(x)\,dx\right]\frac{\Phi_{n,i}^m(y)}{u(x_i)} = \sum_{i=1}^n e_i\frac{\Phi_{n,i}^m(y)}{u(x_i)},\quad |y|\le 1,
\end{equation*}
where we set $e_i=g(x_i)u(x_i)$. Equating the coefficients of the basis $\mathcal{B}$, we get the following linear system of order $n$ in the unknowns $\{d_i\}_{i=1}^n,$ 
\begin{equation*}
    d_i -\nu \left[ \sum_{j=1}^n  d_j\frac{u(x_i)}{u(x_j)} \int_{-1}^1 k(x,x_i)\Phi_{n,j}^m(x) w(x)\,dx\right] =  e_i, \qquad i=1,\ldots, n,
\end{equation*}
which, due to~\eqref{fi}, can be explicitly written as 
\begin{equation*}
    d_i-\nu\  u(x_i)\sum_{j=1}^n d_j \frac{\lambda_j}{u(x_j)} \sum_{k=0}^{n+m-1}\mu_{n,k}^m \,p_k(v^{\alpha,\beta},x_j)\int_{-1}^1 p_k(v^{\alpha,\beta},x)k(x,x_i)w(x)\,dx=e_i, \quad i=1,\ldots,n.
\end{equation*}
Equivalently, 
\begin{equation}\label{sistema_esplicito}
    \sum_{j=1}^n \left[\delta_{i\,j}-\nu  u(x_i) \frac{\lambda_j}{u(x_j)} \sum_{k=0}^{n+m-1}\mu_{n,k}^m \,p_k(v^{\alpha,\beta},x_j)\widetilde{M}_k(x_i)\right]d_j=e_i, \quad i=1,\ldots,n,
\end{equation}
where the integrals 
\begin{equation}\label{eq:Modified_Moments}
    \widetilde{M}_k(y):= \int_{-1}^1 p_{k}(v^{\alpha,\beta},x)k(x,y)w(x)\,dx, \quad k=0,\ldots,n+m-1,
\end{equation}
are the so-called \emph{modified moments}.

Denoting by $\bf{I}_n$ the identity matrix of order $n$, the system \eqref{sistema_esplicito} takes  the following matrix form
\begin{equation}\label{matr_vp}
    (\bm{I}_n - \nu \bm{K}_n) \bm{d}_n = \bm{e}_n.
\end{equation}
where we have set
\begin{equation*}
    \bm{d}_n=[d_1, \ldots, d_n]^T, \qquad \bm{e}_n=[e_1,\ldots,e_n]^T, \qquad \bm{K}_n = \bm{U} \bm{\widetilde{M}} \bm{\Psi} \bm{P} \bm{\Lambda} \bm{U}^{-1},
\end{equation*}
with 
\begin{equation*}
    \begin{gathered}        \bm{U}=\operatorname{diag}\left(u(x_1),\ldots,u(x_n)\right),\quad        \bm{\Lambda}=\operatorname{diag}\left(\lambda_1,\ldots,\lambda_n\right) \in \mathbb{R}^{n \times n},\\
    \bm{\Psi}=\operatorname{diag}\left(\mu_{n,0}^m,\ldots,\mu_{n,n+m-1}^m\right) \in \mathbb{R}^{(n+m) \times (n+m)},\\             \bm{\widetilde{M}}=\left(\widetilde{m}_{ij}\right)_{i,j} \in \mathbb{R}^{n \times (n+m)}, \qquad \widetilde{m}_{ij}=\widetilde{M}_j(x_i), \qquad i=1,\ldots,n, \qquad j=0,\ldots,n+m-1, \\
    \bm{P}=\left(p_{ij}\right)_{i,j} \in \mathbb{R}^{(n+m) \times n}, \qquad p_{ij}=p_i(v^{\alpha,\beta},x_j), \qquad i=0,\ldots,n+m-1,\qquad j=1,\ldots,n.
    \end{gathered}
\end{equation*}
Additional details on the computation of the matrix $\bm{\widetilde{M}}$ are provided in Section \ref{sec:moments} for specific choices of the kernel $k(x,y)$.

We conclude by establishing the well-conditioning of the final linear system~\eqref{matr_vp}. Setting $\bm{N}_n=\bm{I}_n- \nu \bm{K}_n,$ and denoting by $\kappa_\infty(\bm{N}_n):=\left\|\bm{N}_n\right\|_\infty \, \left\|\bm{N}^{-1}_n\right\|_\infty$ its condition number in the infinity norm (or row sum norm), the following result holds true.
\begin{theorem}\label{stability}
	Assume that $|\alpha|=|\beta|=\frac 12 $. Under the assumptions  of Theorem \ref{convergence}, setting 
    \begin{equation*}
        \kappa( I- \nu K):= \|I-\nu K\| \|(I-\nu K)^{-1}\|,
    \end{equation*}
    we have 
  \begin{equation}\label{finale}
      \C^{-1} \kappa(I-\nu K)\le \liminf_{n\to \infty} \kappa_\infty(\bm N_n)\le \limsup_{n\to \infty}\kappa_\infty(\bm N_n)\le \C \kappa(I-\nu K).
  \end{equation}  
\end{theorem}

\section{A comparison with a Lagrange based method}\label{sec:LagrColl}
In this section, we briefly report the projection method introduced in \cite{dbmsiam}, which is based on the Lagrange operator $\mathcal{L}_n(v^{\alpha,\beta})$ in \eqref{eq:Lagrange}. From now on, we refer to it as the \emph{$\mathcal{L}$-method}. As mentioned above, it represents in a certain sense the global method closest to the $VP$-method, also because $V_n^m(v^{\alpha,\beta})$ coincides with $\mathcal{L}_n(v^{\alpha,\beta})$ when $m=1$. However, the applicability of the $\mathcal{L}$-method in \cite{dbmsiam} requires the weight $v^{\alpha,\beta}$ employed in $\mathcal{L}_n(v^{\alpha,\beta})$ to coincide with the weight $w$ defining the operator $K$ in~\eqref{eq:Integral_Operator}.

The application of the Lagrange operator to \eqref{eq:FIE_Compact} yields the discrete equation
\begin{equation}\label{eqfinitelag} 
    (I-\nu \tilde K_n)f_n^L=\tilde{g}_n, \qquad (\tilde K_nf)(y)=\mathcal{L}_n(v^{\alpha,\beta}, Kf,y), \qquad  \tilde{g}_n(y)=\mathcal{L}_n(v^{\alpha,\beta},g,y),
\end{equation}
with
\begin{equation*}
    f_n^L(y)=\sum_{j=1}^n \tilde b_j \frac{\ell_{n,j}(y)}{u(x_j)}\in \PP_{n-1}.
\end{equation*}
The unknown coefficients $\tilde b_i$, $\  i=1,\dots,n,$ are then obtained as the solution of the linear system
\begin{equation}\label{sistema_lagrange}
    \sum_{j=1}^n \left[ \delta_{i\,j} - \nu  u(x_i)  \frac{\lambda_j}{u(x_j)} \sum_{k=0}^{n-1} \,p_k(v^{\alpha,\beta},x_j)M_k(x_i)
   \right] \tilde b_j = e_i, \qquad i=1,2,\dots,n,
\end{equation}
where $e_i=g(x_i)u(x_i), \ i=1,2,\dots,n,$ and
\begin{equation*}
    M_k(y):=\int_{-1}^1 p_k(v^{\alpha,\beta},x) k(x,y) v^{\alpha,\beta}(x)dx, \quad k=0,1,\dots, n-1, \qquad 
\end{equation*}
are the modified moments. The system \eqref{sistema_lagrange} admits the following compact representation
\begin{equation}\label{matr_lagr}
    (\bm{I}_n-\nu \bm{H}_n) \bm{\widetilde{b}}_n= \bm{e}_n,
\end{equation}
where
\begin{equation*}
   \tilde {\bm{b}}_n=[\tilde b_1, \ldots, \tilde b_n]^T, \qquad \bm{e}_n=[e_1,\ldots,e_n]^T, \qquad \bm{H}_n = \bm{U} \bm{M} \bm{Q} \bm{\Lambda} \bm{U}^{-1},
\end{equation*}
\begin{equation*}
    \begin{gathered}
        \bm{M}=\left(m_{ij}\right)_{i,j} \in \mathbb{R}^{n \times n}, \quad m_{ij}=M_j(x_i), \quad i=1,\ldots,n, \; j=0,\ldots,n-1, \\
        \bm{Q}=\left(q_{ij}\right)_{i,j} \in \mathbb{R}^{n \times n}, \quad q_{ij}=p_i(w,x_j), \quad i=0,\ldots,n-1, \; j=1,\ldots,n.
    \end{gathered}
\end{equation*}
The following result, proved in \cite{dbmsiam}, addresses the convergence of the $\mathcal{L}$-method.

\begin{theorem}\label{convergence_lag}
    Let $\mathrm{ker}(I-\nu K)=\{0\}$ and $K: C_u\to Z_r(u)$ be bounded for some $r>0$. For any  $g\in Z_s(u)$ $s\le r$ equation \eqref{fie} has a unique solution $f^*\in Z_s(u)$.  
If in addition 
\begin{equation}\label{ipo_lagrange}
    \max\left\{0,\frac \alpha 2 +\frac 1 4\right\}\leq  \gamma< \min\left\{\frac \alpha 2 +\frac 3 4,\alpha+1\right\}, \quad \max\left\{0,\frac \beta 2 +\frac 1 4\right\}\leq  \delta< \min\left\{\frac \beta 2 +\frac 3 4,\beta+1\right\},\end{equation}
then, for $n$ sufficiently large (say $n>n_0$), also the finite dimensional equation \eqref{eqfinitelag}  admits a unique solution $f_n^{L} \in \mathbb{P}_{n-1}$ and  the following estimate holds
\begin{equation}\label{eqcor1_lag}
        \| (f^*-f_n^{L})\|_{C_u} \le  \C  \log n \frac{\|g \|_{Z_s(u)}}{n^s},  \quad \C \neq \C(n,f^*).
    \end{equation}
\end{theorem}

We note that  Theorem \ref{convergence_lag} is the analogous of  Corollary \ref{cor1}, but its assumptions on the weights $w$ and $u$ are more restrictive (compare \eqref{ipo_lagrange} with \eqref{hp1-inf}--\eqref{hp3-inf}). Moreover, in  \eqref{eqcor1_lag} we have an extra $\log n$-factor. Finally, we remark that  the analogous of Theorem \ref{convergence} cannot be stated for the $\mathcal{L}$-method since, due to the unboundedness of the Lebesgue constants of Lagrange projector, we need $Kf$ to be smooth enough to ensure that
\begin{equation*}
    \lim\limits_{n\to \infty} \log n\  E_n(Kf)_{C_u}=0, \qquad \forall f\in C_u.
\end{equation*}

In conclusion, we point out that the $VP$-method is based on an additional weight $v^{\alpha,\beta}$ that can be different from the weight $w$ in the equation. Consequently, it can happen that the $VP$-method is applicable whereas the $\mathcal{L}$-method is not. For instance, in the case $u=v^{0,0}, \ w=v^{1,1}$, by taking $v^{\alpha,\beta}=w$ none of the assumptions \eqref{ipo_lagrange} and \eqref{hp1-inf}--\eqref{hp3-inf} are satisfied. Hence, the $\mathcal{L}$-method cannot be applied whereas considerable flexibility is available in the implementation of the $VP$-method by choosing the weight $v^{\alpha,\beta}$  according to \eqref{hp1-inf}--\eqref{hp3-inf}.
 
Anyway, in the case that both the $VP$-method and the $\mathcal{L}$-method are applicable with the same weight $v^{\alpha,\beta}$, the numerical tests reveal that the additional logarithmic factor $\log n$ does not affect the maximum errors significantly, but the presence of the free parameter $m$ may improve the pointwise approximation provided by the $VP$-method.

\section{On the computation of the modified moments for some kernels}\label{sec:moments}

In this section, we provide some recurrence relations for some families of modified moments $\widetilde{M}_k(y)$. First, we recall the three-term recurrence relation satisfied by the sequence $ \{p_k(v^{\rho, \, \sigma},y)\},$
\begin{equation}\label{three_term}
    b_k(v^{\rho, \, \sigma})p_{k}(v^{\rho, \, \sigma},y)=(y-a_{k-1}(v^{\rho, \, \sigma}))p_{k-1}(v^{\rho, \, \sigma},y)-b_{k-1}(v^{\rho, \, \sigma})p_{k-2}(v^{\rho, \, \sigma},y), \qquad \ \  k\geq1.
\end{equation}
with $p_{-1}(v^{\rho, \, \sigma})=0,$ $p_0(v^{\rho, \, \sigma})=\left(\int_{-1}^1v^{\rho, \, \sigma}(x) \, dx\right)^{-\frac{1}{2}}=\left(2^{\rho+\sigma+2}\frac{\Gamma(\rho+1)\Gamma(\sigma+1)}{\Gamma(\rho+\sigma+2)}\right)^{-\frac{1}{2}}$ \begin{equation}\label{eq:Coefficients_3terms}
        \begin{array}{ll}
            a_0(v^{\rho, \, \sigma})=\dfrac{\sigma-\rho}{\rho+\sigma+2},    &  b_0(v^{\rho, \, \sigma})=0, \\
            a_1(v^{\rho, \, \sigma})=\dfrac{\sigma^2-\rho^2}{(\rho+\sigma+2)(\rho+\sigma+4)}, &   b_1(v^{\rho, \, \sigma})=\sqrt{\dfrac{4(\rho+1)(\sigma+1)}{(\rho+\sigma+2)^2(\rho+\sigma+3)}}, \\
            a_k(v^{\rho, \, \sigma})=\dfrac{\sigma^2-\rho^2}{(\rho+\sigma+2k)(\rho+\sigma+2+2k)},  &  b_k(v^{\rho, \, \sigma})=\sqrt{\dfrac{4k(\rho+k)(\sigma+k)(\rho+\sigma+k)}{(\rho+\sigma+2k)^2 (\rho+\sigma+1+2k)(\rho+\sigma-1+2k)}},  
        \end{array}
    \end{equation}
for which we refer to \cite[pp.~131–133]{mastromilobook} and references therein.

\subsection{Weakly singular kernel}
A recurrence relation for the modified moments
\begin{equation*}
    M_k(y)=\int_{-1}^{1} p_k(v^{\rho,\sigma},x) \left| x-y \right|^\mu v^{\rho,\sigma}(x) \, dx, \qquad \mu >-1, \quad \mu \neq0, 
\end{equation*}
has been derived in \cite{iwota}. Specifically, 
\begin{equation*}
    b_{k+1}(v^{\rho, \, \sigma})\left( 1+\dfrac{\mu+1}{k+\rho+\sigma+1} \right) M_{k+1}(y)= (y+\gamma_k) M_k(y)+  b_{k}(v^{\rho, \, \sigma})\left( \dfrac{\mu+1}{k}-1 \right) M_{k-1}(y), \quad k\geq1,
\end{equation*}
with
\begin{eqnarray*}
    \gamma_{k}&=&\dfrac{(\rho-\sigma)(\rho+\sigma+2\mu+2)}{(2k+\rho+\sigma)(2k+\rho+\sigma+2)}.
\end{eqnarray*}
The initial moments are
\begin{eqnarray*}
    M_{0}(y)&=&p_0(v^{\rho, \, \sigma})(m_{-}(\rho,\sigma,\mu,y)+m_{+}(\rho,\sigma,\mu,y)),\\
    M_{1}(y)&=&\dfrac{p_0(v^{\rho, \, \sigma})}{b_{1}(v^{\rho, \, \sigma})}(m_{+}(\rho,\sigma,\mu+1,y)-m_{-}(\rho,\sigma,\mu+1,y))+\dfrac{1}{b_{1}(v^{\rho, \, \sigma})}(y-a_{0})M_{0}(y)
\end{eqnarray*}
and
\begin{eqnarray*}
    m_{-}(\rho,\sigma,\mu,y)&=&2^{\rho}(1+y)^{\sigma+\mu+1}B(1+\sigma,1+\mu) \, _{2}F_{1} \left( -\rho,1+\sigma,\sigma+\mu+2,\dfrac{1+y}{2} \right),\\
    m_{+}(\rho,\sigma,\mu,y)&=&2^{\sigma}(1-y)^{\rho+\mu+1}B(1+\rho,1+\mu) \, _{2}F_{1} \left( -\sigma,1+\rho,\rho+\mu+2,\dfrac{1-y}{2} \right).
\end{eqnarray*}
In the case of the first kind Chebyshev polynomials, i.e. for modified moments of the type
\begin{equation*}
    M_k^T(y)=\int_{-1}^{1} p\left(v^{-\frac 1 2,-\frac 1 2},x\right) |x-y|^\mu v^{\rho,\sigma}(x) \, dx, 
\end{equation*}
the following recurrence relation has been derived in \cite{KangMom} 
\begin{eqnarray*}
    (\rho+\sigma+\mu+k+3)M_{k+2}^T(y)=-\sqrt{\frac 2 \pi}&\bigg[& 2\left( (1-y)(\rho+1)-(1+y)(\sigma+1)-ky \right)M_{k+1}^T(y)\\
    &+&2\left( (1-2y)(\rho+1)+(1+2y)(\sigma+1)-(\mu+1) \right)M_k^T(y)\\
    &+&2 \left( (1-y)(\rho+1)-(1+y)(\sigma+1)+ky\right)M_{k-1}^T(y)\\
    &+& (\rho+\sigma+\mu-k+3)M_{k-2}(y)\bigg], \quad k\geq 1.
\end{eqnarray*}
The required initial moments are 
\begin{eqnarray*}
    &M_0^T(y)&=m_{-}(\rho,\sigma,\mu,y)+m_{-}(\sigma,\rho,\mu,-y),\\
    &M_1^T(y)&=m_{-}(\rho,\sigma+1,\mu,y)+m_{-}(\sigma+1,\rho,\mu,-y)-M_0^T(y),\\
    &M_2^T(y)&=2[m_{-}(\rho+2,\sigma,\mu,y)+m_{-}(\sigma,\rho+2,\mu,-y)]\\ && +4[m_{-}(\rho,\sigma+1,\mu,y)+m_{-}(\sigma+1,\rho,\mu,-y)]-7M_0^T(y),
\end{eqnarray*}
due to the fact that $M_{-1}^T(y)=M_{1}^T(y)$ and $M_{-2}^T(y)=M_{2}^T(y)$.

Finally, after computing  $\left\{M_k(y)\right\}_k$, the modified moments  $\left\{\widetilde{M}_k(y)\right\}_{k}$ are obtained by  means of a simple basis transformation matrix, as detailed in what follows. 

\begin{proposition}\label{prop_transform}
    Setting $$\mathcal{M}_{\mathbf{n+m}}=[M_0(y),M_1(y),\dots, M_{n+m-1}(y)]^T,$$ and 
    $$\widetilde{\mathcal{M}}_{\mathbf{n+m}}=[\widetilde M_0(y), \widetilde M_1(y),\dots,\widetilde  M_{n+m-1}(y) ]^T,$$
    the following transformation rule holds 
    \begin{equation}\label{matrix_moment_transform}  
        \mathcal{\widetilde{M}}_{\mathbf{n+m}}=\bm{C}\mathcal{M}_{\mathbf{n+m}},
    \end{equation}
    where the matrix $\bm{C}\in \mathbb{R}^{(n+m) \times (n+m)},$ is defined as 
    \begin{equation*}
        \bm{C} = \bm{P}_v \bm{\Lambda}_{n+m} \bm{P}_w,
    \end{equation*}
     with
    \begin{equation*}
        \begin{gathered}
            \bm{P}_v=\left(p_{k,h}\right)_{k,h} \in \mathbb{R}^{(n+m) \times (n+m)}, \quad p_{k,h}=p_k(v^{\alpha,\beta},x_{n+m,h}(w)), \quad k=0,\ldots,n+m-1, \; h=1,\ldots,n+m,\\
            \\ \bm{\Lambda}_{n+m}=\operatorname{diag}\left(\lambda_{n+m,1}(w),\ldots,\lambda_{n+m,n+m}(w)\right) \in \mathbb{R}^{(n+m) \times (n+m)},\\
            \\ \bm{P}_w=\left(\overline{p}_{h,j}\right)_{h,j} \in \mathbb{R}^{(n+m) \times (n+m)}, \quad \overline{p}_{h,j}=p_j(w,x_{n+m,h}(w)), \quad h=1,\ldots,n+m, \; j=0,\ldots,n+m-1. 
        \end{gathered}
    \end{equation*}
\end{proposition}

Similarly, $\left\{\widetilde{M}_k(y)\right\}_{k}$ can be obtained by $\left\{{M}_k^T(y)\right\}_{k}.$

\subsection{Logarithmic kernel}
Let us now consider modified moments of the following type
\begin{equation}\label{eq:mom_log}
    M_k(y)=\int_{-1}^1 \log\lvert x-y\rvert \, p_k (v^{\rho,\sigma},x) \, v^{\rho,\sigma}(x) \, dx, 
\end{equation}
for which the following result holds.
\begin{proposition} 
\label{thm:Momenti_log}
   The modified moments defined in \eqref{eq:mom_log} satisfy the three-term recurrence relation
    \begin{equation}\label{eq:Ricorrenza_Momenti_Logaritmici}
        \dfrac{b_k(v^{\rho+1,\sigma+1})}{c_{k+1}}M_{k+1}(y)=\dfrac{(y-a_{k-1}(v^{\rho+1,\sigma+1}))}{c_k}M_k(y)-\dfrac{b_{k-1}(v^{\rho+1,\sigma+1})}{c_{k-1}}M_{k-1}(y), \qquad \text{for} \ k\geq 2,
    \end{equation}
    with 
    \begin{equation*}
        \begin{split}
            \qquad M_1(y)=&c_1\left(2^{\rho+\sigma+4}\frac{\Gamma(\rho+2)\Gamma(\sigma+2)}{\Gamma(\rho+\sigma+4)}\right)^{-\frac{1}{2}} \left(v^{\rho+1, \sigma+1}(y)\, \pi \cot(\pi \rho) - \frac{2^{\rho + \sigma+2} \Gamma(\rho+1) \Gamma(\sigma + 2)}{\Gamma(\rho + \sigma + 3)} \right. \times\\
            &\left. {}_2F_1\left( -\rho - \sigma-2,\ 1;\ - \rho;\ \frac{1 - y}{2} \right)\right),
          \end{split}
    \end{equation*}        
    
    \begin{eqnarray*}
        M_2(y)&=&c_2\left(2^{\rho+\sigma+4}\frac{\Gamma(\rho+2)\Gamma(\sigma+2)}{\Gamma(\rho+\sigma+4)}\right)^{-\frac{1}{2}}\sqrt{\dfrac{(\rho+\sigma+4)^2(\rho+\sigma+5)}{4(\rho+2)(\sigma+2)}}\left[\left(2^{\rho+\sigma+4}\frac{\Gamma(\rho+2)\Gamma(\sigma+2)}{\Gamma(\rho+\sigma+4)}\right)\right. \\
        &+& \left. \left(y-\dfrac{(\sigma-\rho)(\sigma+\rho-2)}{(\rho+\sigma+4)(\rho+\sigma+6)}\right) \right]\left(v^{\rho+1, \sigma+1}(y)\, \pi \cot(\pi \rho) - \frac{2^{\rho + \sigma+2} \Gamma(\rho+1) \Gamma(\sigma + 2)}{\Gamma(\rho + \sigma + 3)} \right. \times\\
        &&\left. {}_2F_1\left( -\rho - \sigma-2,\ 1;\ - \rho;\ \frac{1 - y}{2} \right)\right)
    \end{eqnarray*}
    where ${}_2F_1$ denotes the Hypergeometric function, and 
    \begin{equation}\label{eq:Coefficients_Log}
        c_k=\left(\sqrt{k(k+\rho+\sigma+1)}\right)^{-1},\ k\geq1.
    \end{equation}
    Moreover, if the weight exponents satisfy the conditions
    \begin{equation}\label{eq:Cond_Pesi_Log}
        \rho\in (-1,0), \quad \sigma\in (-1,0) \quad \text{and} \quad \rho+\sigma=-1,
    \end{equation}
    then the moments admit the explicit representation
    \begin{equation}\label{eq:Momenti_Log_Particolari}
        M_k(y)=c_k \pi \left( \,\cot(\pi \rho) \, p_{k-1}(v^{\rho+1, \,\sigma+1},y)+\csc(\pi\rho) \, ) \, p_{k+\rho+\sigma+1}(v^{-\rho-1, \,-\sigma-1},y)  \right), \qquad \text{for} \ k\geq 1.
    \end{equation}
\end{proposition}
\begin{remark}
    We point out that the initial moment $M_0(y)$  can be found in \cite{gradshteyn2007table} for some choices of $\rho,\sigma$.
   \end{remark}

In this case also, after computing $\left\{M_k(y)\right\}_k$, the modified moments  $\left\{\widetilde{M}_k(y)\right\}_{k}$ are obtained by  using the transformation given in Proposition \ref{prop_transform}.
 
\subsection{Highly oscillating kernel}

Finally, we address modified moments of  the following type
\begin{equation*}
    M_k=\int_{-1}^{1} p_k\left(v^{-\frac 1 2,-\frac 1 2},x\right) e^{i\omega x} v^{\rho,\sigma}(x)\,dx.
\end{equation*}
By separating the kernel $e^{i \omega x}$ into its real and imaginary parts, we obtain moments of the type
\begin{equation*}
     \overline{M}_k=\int_{-1}^{1} p_k\left(v^{-\frac 1 2,-\frac 1 2},x\right) \eta(\omega x) v^{\rho,\sigma}(x)\,dx, \quad \eta(\omega x)=\begin{cases}\cos{(\omega x)},\\\sin{(\omega x)}.\\\end{cases}
\end{equation*}
The following five-term recurrence relation holds {\cite{PiessBraOsc}}
\begin{equation*}
    i \omega M_{k+2}=-2(k+\rho+\sigma+2)M_{k+1}+2(2\rho-2\sigma+i\omega)M_k+2(k-\rho-\sigma-2)M_{k-1}-i\omega M_{k-2}, \quad k\geq2.
\end{equation*}
Denoting by
\begin{equation*}
    G(\rho,\sigma,\omega):=2^{\rho+\sigma+1}e^{-i\omega}\frac{\Gamma(\rho+1)\Gamma(\sigma+1)}{\Gamma(\rho+\sigma+2)} \,_1F_1(\rho+1,\rho+\sigma+2,2i\omega),
\end{equation*}
the initial moments are
\begin{eqnarray*}
    M_0&=&G(\rho,\sigma,\omega),\\
    M_1&=&G(\rho+1,\sigma,\omega)-M_0,\\
    M_2&=&\frac{i}{\omega}[2(\rho+\sigma+2)M_1-(2\rho-2\sigma+i\omega)M_0],\\
    M_3&=&\frac{i}{\omega}[2(\rho+\sigma+3)M_2-(4\rho-4\sigma+i\omega)M_1+2(\rho+\sigma+1)M_0].
\end{eqnarray*}
This recurrence relation is forward stable for high values of $\omega$. On the other hand, for $n>2|\omega|$ there is a significant loss of precision, so we suggest using $n \sim |\omega|$ for better results.

Once the sequence $\left\{M_k\right\}_k$ has been computed, the modified moments  $\left\{\widetilde{M}_k\right\}_{k}$ are retrieved through the transformation detailed in Proposition \ref{prop_transform}.

\section{Numerical tests}\label{sec:numtest}
In this section, we provide some numerical examples to assess  the accuracy of $VP$-method  and compare its performance with the $\mathcal{L}$-method of \cite{dbmsiam}. Among the extensive set of numerical experiments performed, we selected some representative test cases involving different kernels $k$, various levels of smoothness of the function $g$ and different combinations of the weights $u,$ $w$ and $v^{\alpha,\beta}$. Here, the superscripts $VP$ and $\mathcal{L}$ denote quantities associated with the $VP$-method and the $\mathcal{L}$-method, respectively.

Given a sufficiently fine uniform grid $Y$ of points in $[-1,1]$, we compute
\begin{equation*}
    \varepsilon_n^{VP}(f;y) = \left|\left(f^*(y)-f_n^m(y)\right)u(y)\right|, \qquad \varepsilon_n^L(f;y) = \left|\left(f^*(y)-f_n^L(y)\right)u(y)\right|,
\end{equation*}
with $f^*$ denoting a reference solution, given by the exact solution when available, or by the $VP$-method with $n=1024$ and $\theta=0.5$ otherwise. 
Furthermore, we compute the maximum absolute errors 
\begin{equation*}
    \hat\varepsilon_n^{VP}(f) = \max_{y\in Y}\varepsilon_n^{VP}(f;y),\qquad \hat\varepsilon_n^L(f) = \max_{y\in Y}\varepsilon_n^L(f;y),
\end{equation*}
and the mean absolute errors  
\begin{equation*}
    \overline\varepsilon_n^{VP}(f) = \frac{1}{|Y|} \sum_{y \in Y}\varepsilon_n^{VP}(f;y),\qquad \overline\varepsilon_n^L(f) = \frac{1}{|Y|} \sum_{y \in Y} \varepsilon_n^L(f;y).
\end{equation*}
Finally, we report the condition numbers in the infinity norm of the linear systems \eqref{matr_vp} and \eqref{matr_lagr}, associated with the $VP$- and $\mathcal{L}$-methods, respectively, namely,
\begin{eqnarray*}
    \kappa_\infty(\bm{N}_n)&:=&\left\|\bm{N}_n\right\|_\infty \, \left\|\bm{N}^{-1}_n\right\|_\infty, \qquad \bm{N}_n=\bm{I}_n- \nu \bm{K}_n,\quad VP\text{-method}\\
    \kappa_\infty(\bm{T}_n)&:=&\left\|\bm{T}_n\right\|_\infty \, \left\|\bm{T}^{-1}_n\right\|_\infty, \qquad \bm{T}_n=\bm{I}_n- \nu \bm{H}_n, \quad \mathcal{L}\text{-method}.
\end{eqnarray*}

Regarding the $VP$-method, the experimentation has been carried out for different values of the additional parameter $m=\lfloor n\theta\rfloor$, varying $\theta$ in the set $\Theta := \{0.1,0.2, \ldots,0.9\}$, but showing only a selection of the results achieved in the tests. In this regard, details on the values of $\theta $ used will be provided in each example.   

All computations were performed using double precision using MATLAB R2025b on a 2020 M1 MacBook Pro (16 GB RAM) running macOS.

\vspace{0.3cm}

\begin{example}\label{ex3}
Let us consider the following equation
    \begin{equation}\label{eq3_ese3}
        f(y)-\frac{1}{2\pi}\int_{-1}^{1} f(x) |x-y|^{-\frac 1 5} \sqrt[4]{(1-x^2)^3}dx = e^{\cos{2y}}+\left| y- \frac{3}{5}\right|^{\frac 5 2},
    \end{equation}
    whose exact solution is unknown. Here $w=v^{\frac 3 4,\frac 3 4}$ and with $\gamma=\delta=\frac{7}{10}$ equation \eqref{eq3_ese3} the function  $g\in Z_{\frac{5}{2}}(u)$.
    On the other hand, since the assumptions of Lemmas \ref{boundedness} and  \ref{compattezza_zyg} are satisfied, 
    the equation  admits a unique solution in  $Z_{\frac{4}{5}}(u)$, 
    Moreover, for $v^{\alpha,\beta}=w$, the hypotheses that guarantee the convergence of the $VP$-method and the $\mathcal{L}$-method are both satisfied, and their theoretical orders of convergence are  $\mathcal{O}(n^{-4/5})$ and $\mathcal{O}(\log n \ n^{-4/5})$, respectively. 
    
    Table~\ref{tab3} shows that, the maximum absolute errors, the condition numbers of the linear systems  and the mean absolute errors, are comparable in  $VP$-method and $\mathcal{L}$-method. However, the $VP$-method performs slightly better in terms of mean absolute errors. This behavior is confirmed by the plots of the pointwise errors for $\theta^*=0.1$, which represents the best choice of $\theta$, giving the lower maximum absolute errors,   and $\theta=0.5$ in Figure~\ref{fig:ex3}.

\begin{table}[htp!]
    \centering
    \begin{tabular}{c|ccc|ccc}
        $n$ & $\hat\varepsilon_n^{VP}(f)$ & $\overline{\varepsilon}_n^{VP}(f)$ & $\kappa_\infty(\bm{N}_n)$ &  $\hat\varepsilon_n^L(f)$ & $\overline{\varepsilon}_n^L(f)$ & $\kappa_\infty(\bm{T}_n)$  \\ \hline
        8    & 8.18e-03 & 4.04e-03 & 1.82  & 8.18e-03 & 4.04e-03 & 1.82  \\
        16   & 2.36e-04 & 4.83e-05 & 1.96  & 2.67e-04 & 5.61e-05 & 1.96  \\
        32   & 4.99e-05 & 4.57e-06 & 2.04  & 4.82e-05 & 6.33e-06 & 2.04  \\
        64   & 9.42e-06 & 4.38e-07 & 2.09  & 9.60e-06 & 7.70e-07 & 2.09  \\
        128  & 1.34e-06 & 5.10e-08 & 2.11  & 1.37e-06 & 1.01e-07 & 2.11  \\
        256  & 2.97e-07 & 4.01e-09 & 2.13  & 2.94e-07 & 8.99e-09 & 2.13  \\
        512  & 5.17e-08 & 3.28e-10 & 2.14    & 5.70e-08 & 5.34e-10 & 2.14      \\
    \end{tabular}
    \caption{Example \ref{ex3}: Maximum and mean absolute errors comparison between the $VP$-method with $\theta^*=0.1$ and the $\mathcal{L}$-method.}
    \label{tab3}
\end{table}

\begin{figure}[htbp]
	
   \begin{subfigure}{0.45\textwidth}
	   	\centering
	   	\includegraphics[width=\linewidth]{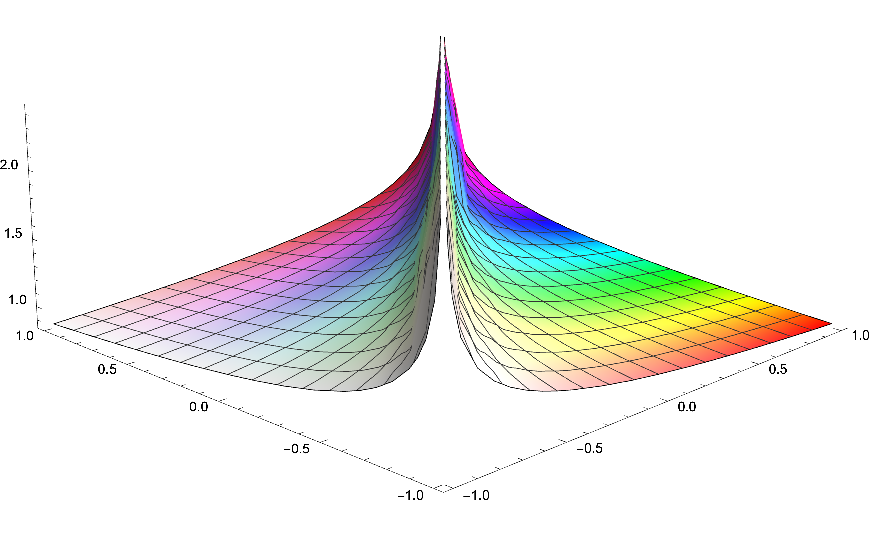}
	   	\caption{Plot of the weakly singular kernel $k(x,y)=|x-y|^{-1/5}$.}
   \end{subfigure}
    \hfill
    \begin{subfigure}{0.45\textwidth}
        \centering
        \includegraphics[width=\linewidth]{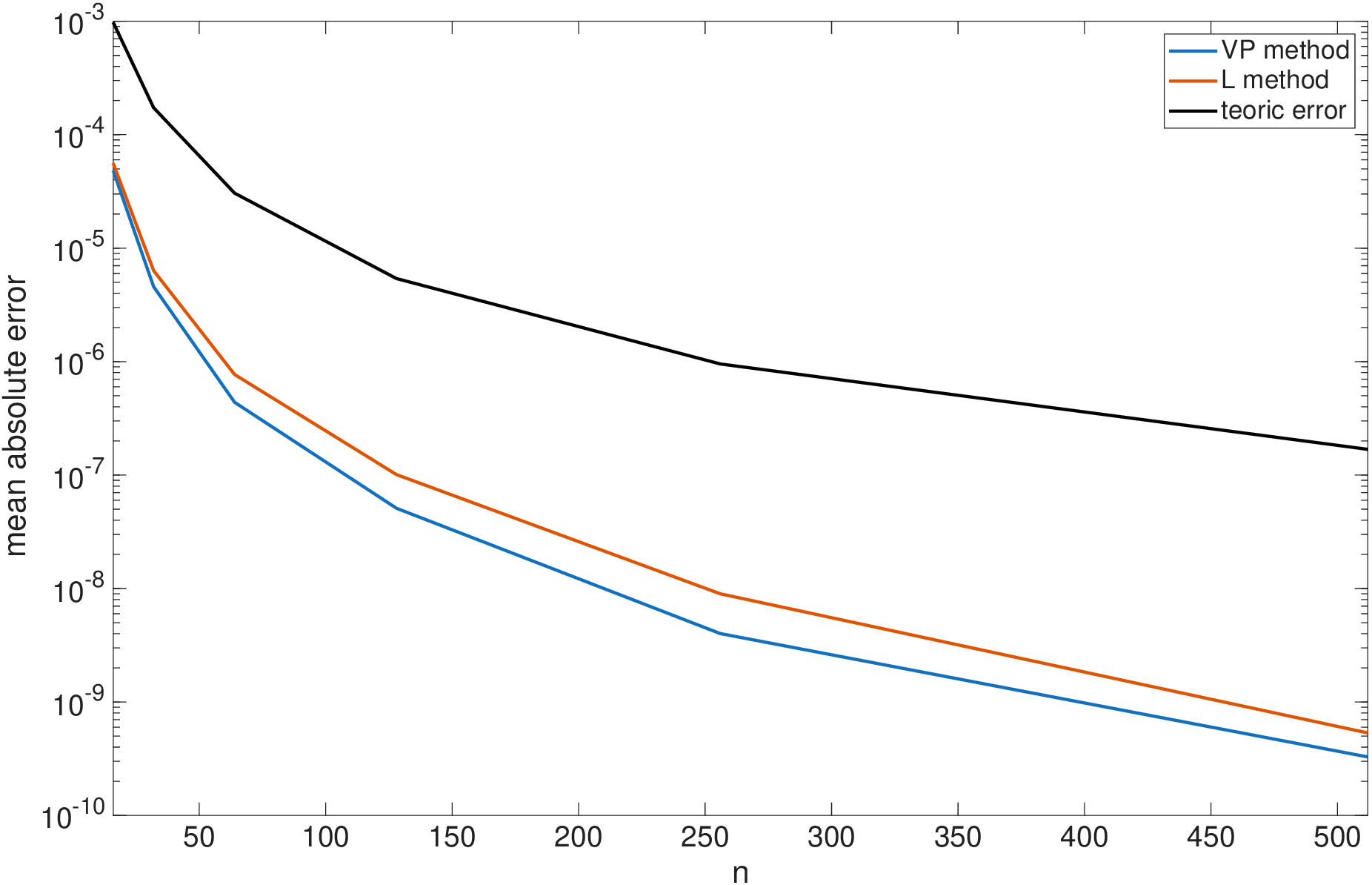}
        \caption{Benchmark analysis of the mean absolute error for the VP-method and the L-method.}
    \end{subfigure}
    
    \vspace{0.5cm}
    
    \begin{subfigure}{0.45\textwidth}
    	\centering
    	\includegraphics[width=\linewidth]{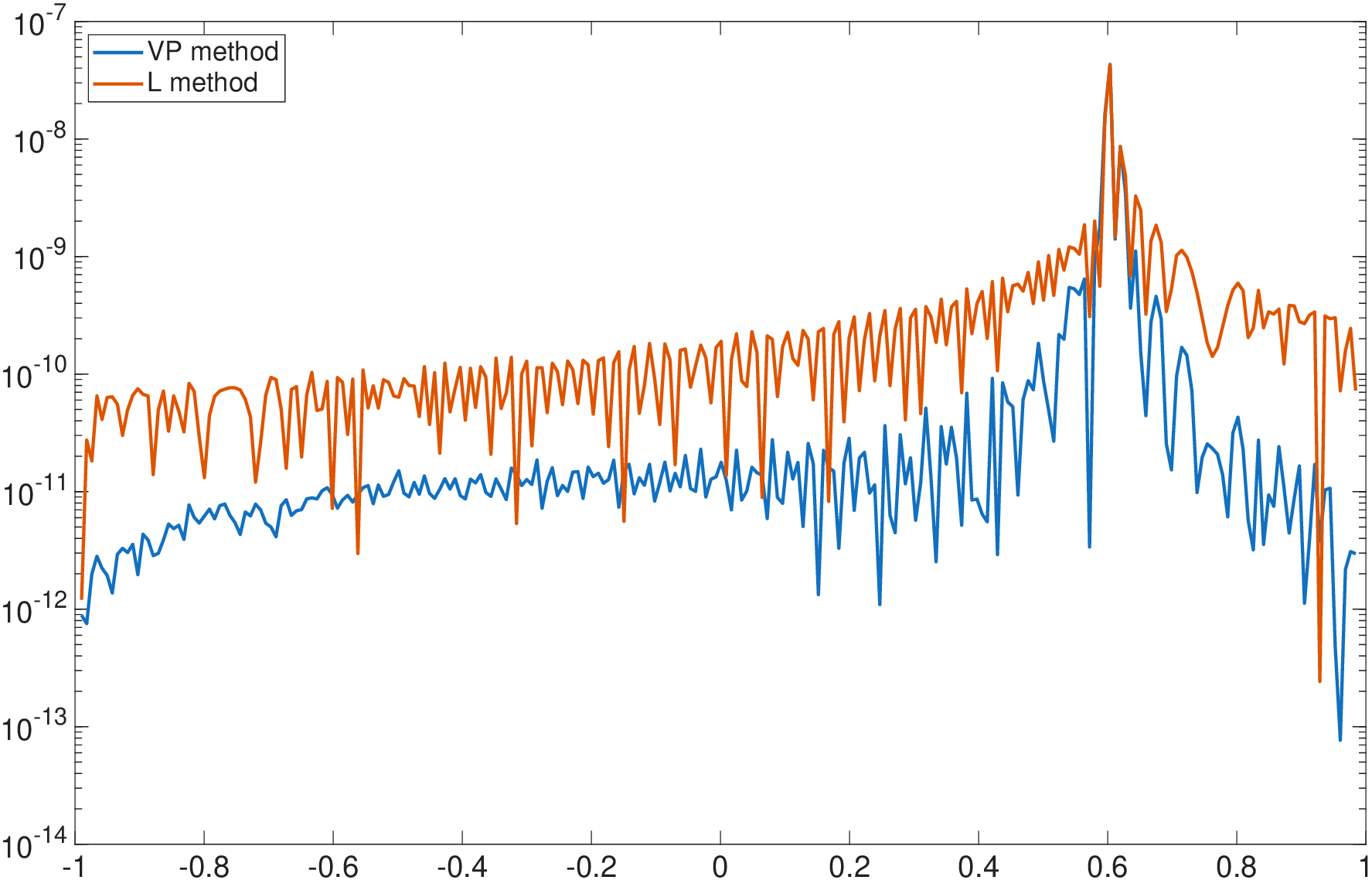}
    	\caption{Plot of the punctual absolute errors for $\theta=0.1$.}
    \end{subfigure}
    \hfill
    \begin{subfigure}{0.45\textwidth}
    	\centering
    	\includegraphics[width=\linewidth]{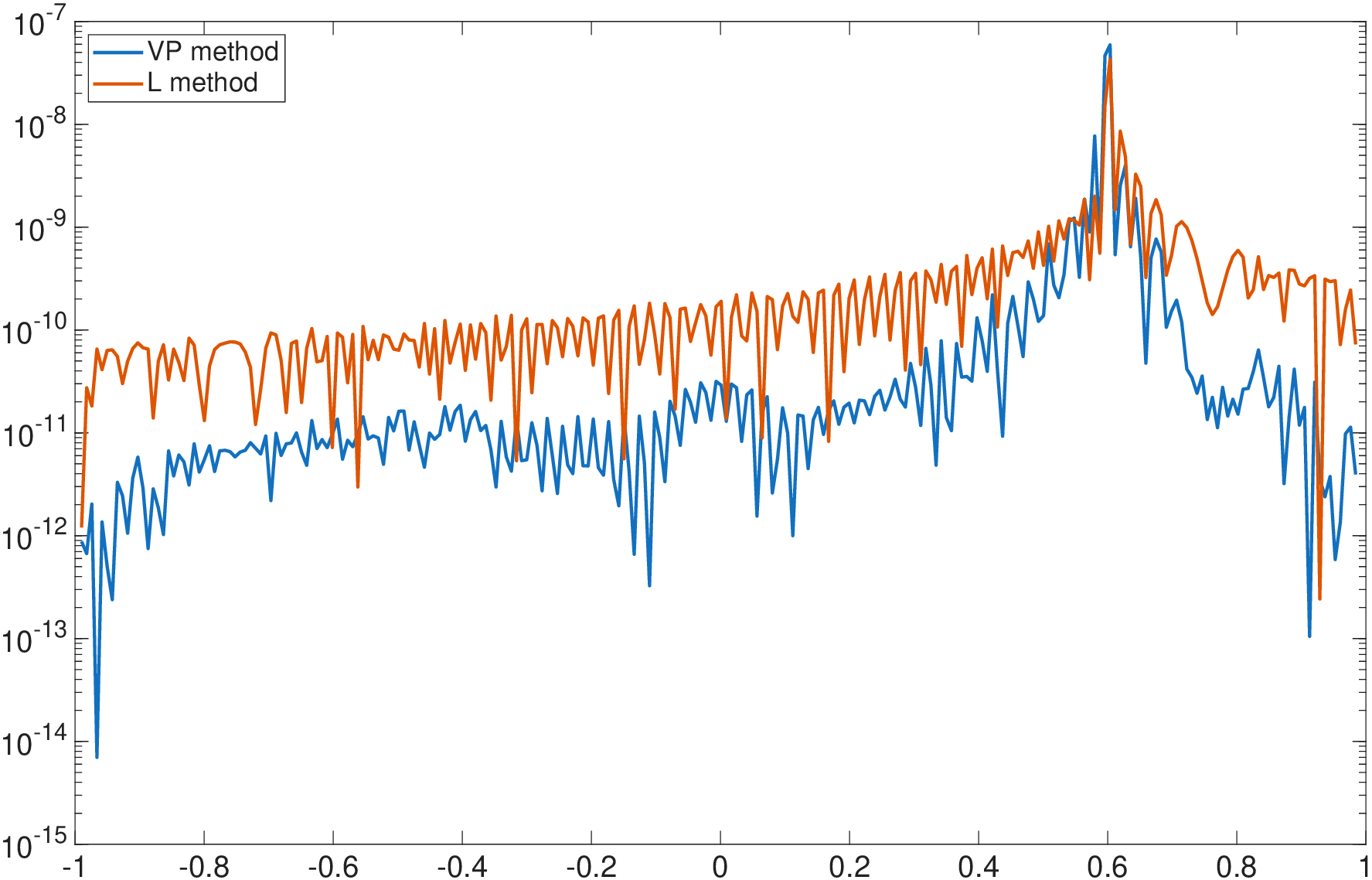}
    	\caption{Plot of the punctual absolute errors for $\theta=0.5$.}
    \end{subfigure}

    \caption{Numerical results of Example \ref{ex3}.}
    \label{fig:ex3}
\end{figure}
\end{example}

\begin{example}\label{ex2}
Let us consider the following equation
\begin{equation}\label{eq_ese2}
	f(y)+\frac 3 8\int_{-1}^{1} f(x)|x-y|^{-\frac 1 2} dx = g(y),
\end{equation}
where $g(y)=|y|+\frac 1 4 \left(\sqrt{1-y}+\sqrt{1+y}+2 y(\sqrt{1-y}-\sqrt{1+y})+4 |y|^{\frac 3 2}\right)$ and  $w=1$. Here the exact solution is $f(y)=|y|$.
Fixing $\gamma=\delta=0$,  $g\in W_1$, and by Lemmas \ref{boundedness} and  \ref{compattezza_zyg}, equation \eqref{eq_ese2} admits a unique solution in $Z_{\frac 1 2}$.  Moreover, for  $\alpha=\beta=0$,  the assumptions that ensure the convergence of the $VP$-method are satisfied, while those of the $\mathcal{L}$-method are not, reason why we implemented only the $VP$-method. 

Table~\ref{tab2} reports the maximum absolute errors and the condition numbers obtained for three different choices of $\theta$. The  numerical values confirm the theoretical expectations, since for all the choices of $\theta$ the errors behave as $\mathcal{O}(n^{-1/2})$, and the linear systems are well conditioned.

\begin{table}[htp!]
	\centering
	\begin{tabular}{c|cc|cc|cc}
		& \multicolumn{2}{c|}{$\theta=0.1$} &\multicolumn{2}{c|}{$\theta=0.3$} & \multicolumn{2}{c}{ $\theta=0.6$}  \\ \hline
		$n$ & $\hat\varepsilon_n^{VP}(f)$ & $\kappa_\infty(\bm{N}_n)$ & $\hat\varepsilon_n^{VP}(f)$ & $\kappa_\infty(\bm{N}_n)$ & $\hat\varepsilon_n^{VP}(f)$ & $\kappa_\infty(\bm{N}_n)$ \\ \hline
		8    & 1.19e-01 & 2.75 & 1.27e-01 & 2.74 & 1.48e-01 & 2.72 \\
		16   & 5.90e-02 & 3.14 & 6.17e-02 & 3.13 & 7.19e-02 & 3.12 \\
		32   & 2.94e-02 & 3.37 & 3.05e-02 & 3.37 & 3.59e-02 & 3.36 \\
		64   & 1.43e-02 & 3.50 & 1.49e-02 & 3.50 & 1.74e-02 & 3.50 \\
		128  & 6.72e-03 & 3.57 & 7.01e-03 & 3.57 & 8.20e-03 & 3.57 \\
		256  & 2.93e-03 & 3.61 & 3.07e-03 & 3.61 & 3.65e-03 & 3.61 \\
		512  & 1.07e-03 & 3.66 & 1.13e-03 & 3.66 & 1.39e-03 & 3.67 \\
		1024 & 2.19e-04 & 3.76 & 2.37e-04 & 3.76 & 3.10e-04 & 3.76 \\
	\end{tabular}
	\caption{Example \ref{ex2}: Maximum absolute errors of the $VP$-method with $\theta \in \{0.1,0.3,0.6\}$.}
	\label{tab2}
\end{table}
\end{example}

\begin{example}\label{ex1}
Let us consider the following equation
\begin{equation}\label{eq_ese1}
	f(y)-\frac 1 8 \int_{-1}^{1} f(x)|x-y|^{-\frac 1 4}(1-x^2)^{\frac 3 4} dx = g(y)
\end{equation}
\begin{eqnarray*}
g(y)&=&\frac{4}{1155}(11(-(1 - y)^{3/4} + (1 + y)^{3/4}) + y (-17 ((1 - y)^{3/4} + (1 + y)^{3/4}) + 4 y (3 ((1 - y)^{3/4} 
\\ & - & (1 + y)^{3/4})+   4 y ((1 - y)^{3/4} + (1 + y)^{3/4}))))
\end{eqnarray*}
where the exact solution is $f(y)= y (1-y^2)^{\frac 1 4}$.  Fixing $\gamma=\delta=\frac{5}{8}$,  the function $g\in Z_{\frac{7}{4}}(u)$, and by Lemmas \ref{boundedness} and  \ref{compattezza_zyg}, equation \eqref{eq_ese1} admits a unique solution in $Z_{\frac{3}{4}}(u)$.  Moreover, for  $\alpha=\beta=\frac 3 4$, that is $v^{\alpha,\beta}=w$, the assumptions ensuring the convergence of the $VP$-method and of the $\mathcal{L}$-method are both satisfied and the theoretical orders of convergence are  $\mathcal{O}(n^{-3/4})$ and $\mathcal{O}(\log n \ n^{-3/4})$, respectively.  Table~\ref{tab1} reports the maximum absolute error attained by the $VP$-method with $m=\lfloor \theta^* n \rfloor$, where $\theta^*=0.1$  yields the best performance of the $VP$-method. 

\begin{table}[htbp]
    \centering
    \begin{tabular}{c|cc|cc}
        $n$ & $\hat\varepsilon_n^{VP}(f)$ & $\kappa_\infty(\bm{N}_n)$ & $\hat\varepsilon_n^L(f)$ & $\kappa_\infty(\bm{T}_n)$ \\ \hline
        8    & 9.79e-03 & 1.60 & 9.79e-03 & 1.60 \\
        16   & 1.89e-03 & 1.70 & 1.79e-03 & 1.70 \\
        32   & 2.89e-04 & 1.77 & 2.66e-04 & 1.77 \\
        64   & 2.38e-05 & 1.81 & 1.96e-05 & 1.81 \\
        128  & 2.73e-06 & 1.83 & 2.04e-06 & 1.83 \\
        256  & 1.10e-07 & 1.85 & 1.97e-07 & 1.85 \\
        512  & 6.12e-09 & 1.85 & 7.46e-09 & 1.85 \\
        1024 & 8.96e-10 & 1.86 & 9.20e-10 & 1.86 \\
    \end{tabular}
    \caption{Example \ref{ex1}: Maximum absolute errors comparison between the $VP$-method with $\theta^*=0.1$ and the $\mathcal{L}$-method.}
	\label{tab1}
\end{table}
\end{example}

\begin{example}\label{ex8}
Let us consider the following equation
\begin{equation}\label{eq_ese8}
	f(y)+\frac 1 5 \int_{-1}^{1} f(x)|x-y|^{\frac 2 3}\sqrt{1-x^2} dx = \sin{\left| x- \frac 1 2 \right|^{\frac 7 2}}.
\end{equation} 
Fixing $\gamma=\delta=\frac{1}{2}$, the function $g\in Z_{\frac{7}{2}}(u)$, and by Lemmas \ref{boundedness} and  \ref{compattezza_zyg}, equation \eqref{eq_ese8} admits a unique solution in $Z_{\frac{5}{3}}(u)$.  Moreover, for  $\alpha=\beta=\frac 1 2$, that is $v^{\alpha,\beta}=w$, the assumptions ensuring the convergence of the $VP$-method and of the $\mathcal{L}$-method are both satisfied and the theoretical orders of convergence are  $\mathcal{O}(n^{-5/3})$ and $\mathcal{O}(\log n \ n^{-5/3})$, respectively.  Table~\ref{tab8} reports the maximum absolute error attained by the $VP$-method with $m=\lfloor \theta^* n \rfloor$, where $\theta^*=0.3$  yields the best performance of the $VP$-method. 
Although the maximum absolute errors of the two methods are comparable, the pointwise errors shown in Figure~\ref{fig:ex8} indicate that the $VP$-method provides higher accuracy. As $\theta$ approaches $1$, the performance deteriorates, since the operator $V_n^m(v^{\alpha,\beta})$ approaches the Fej\'er operator.

\begin{figure}[htbp]
    \centering
    \begin{subfigure}{0.45\textwidth}
        \centering
        \includegraphics[width=\linewidth]{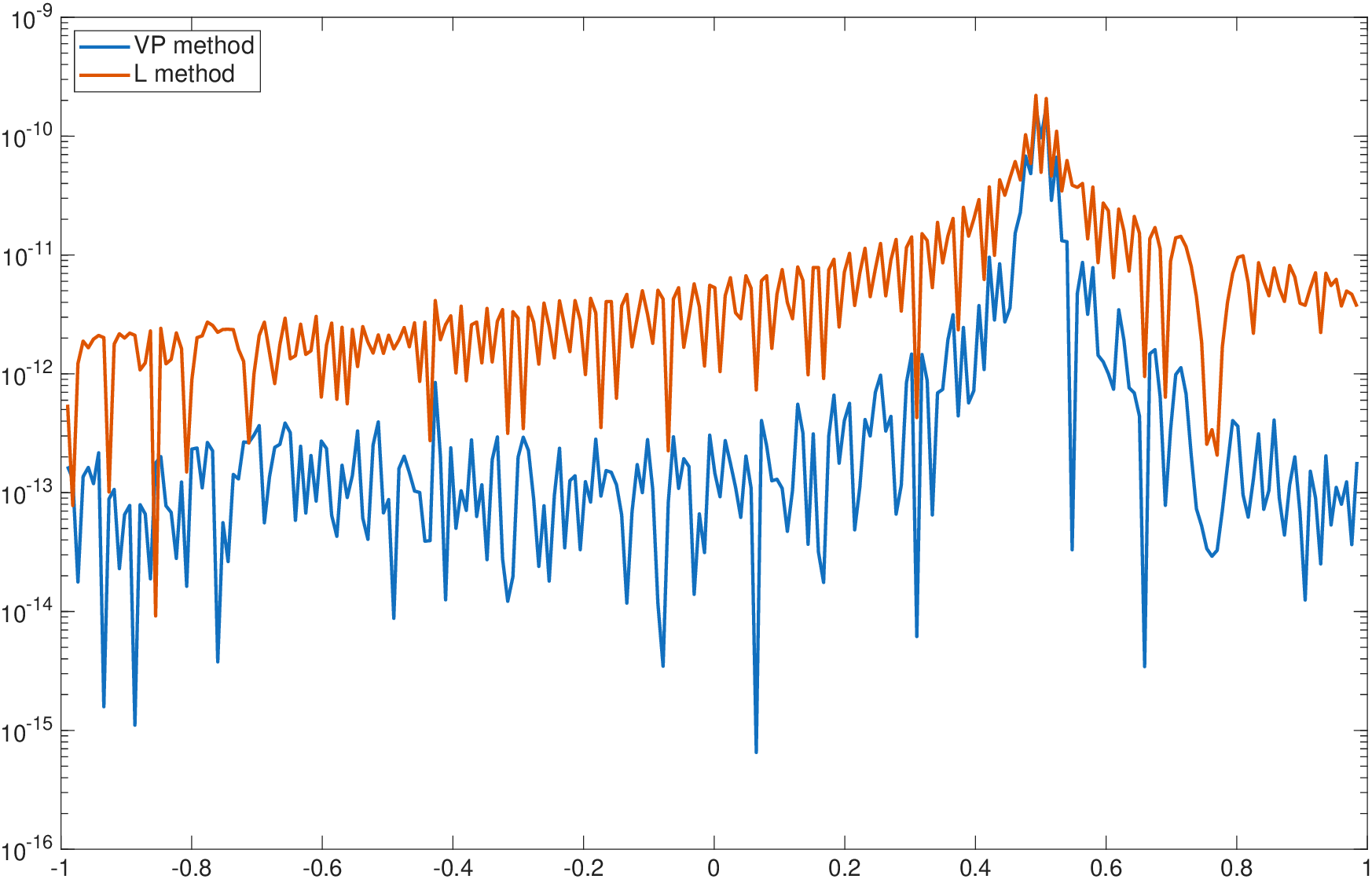}
        \caption{Plot of the punctual absolute errors for $\theta=0.1$.}
    \end{subfigure}
    \hfill
    \begin{subfigure}{0.45\textwidth}
        \centering
        \includegraphics[width=\linewidth]{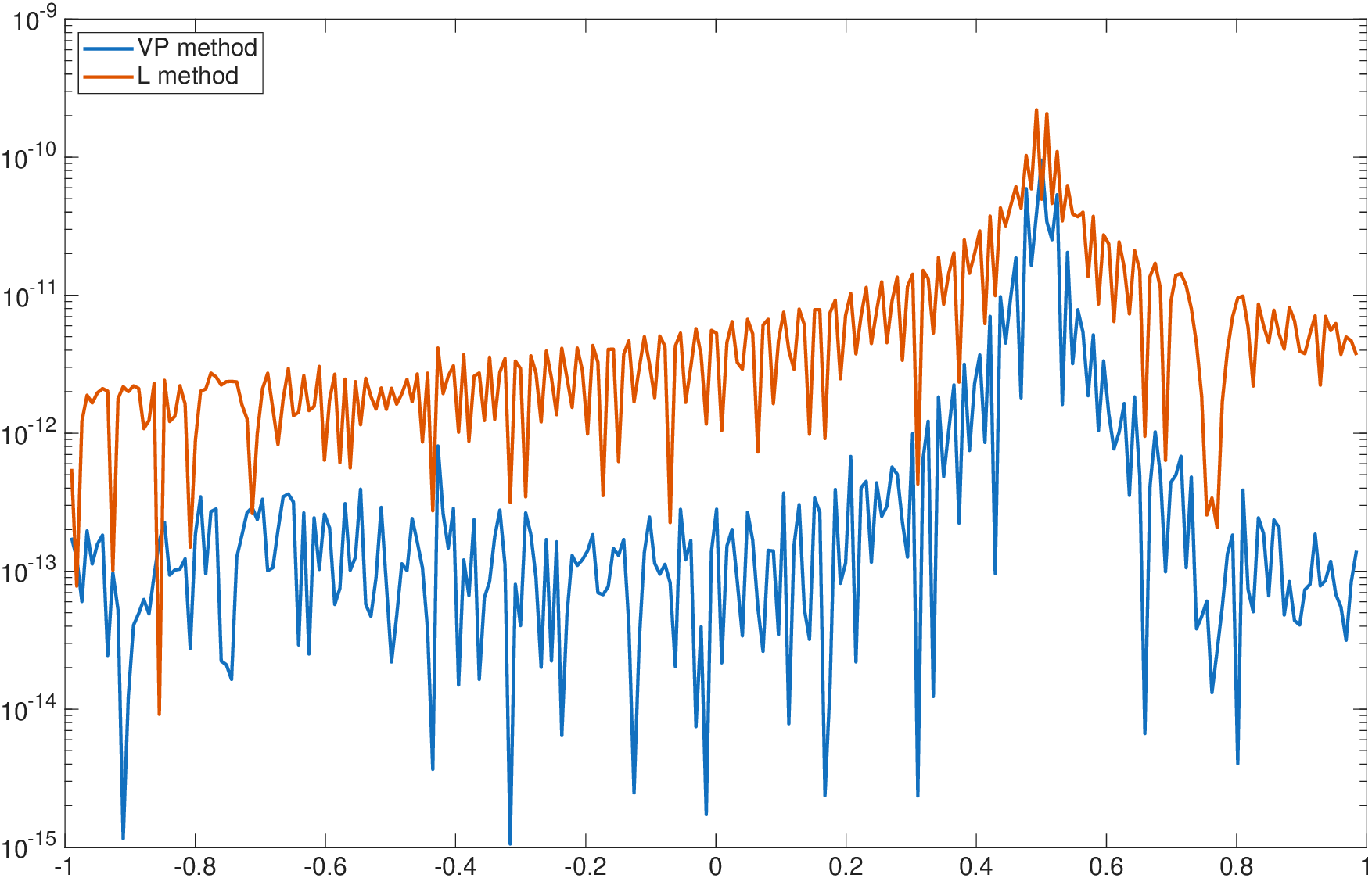}
        \caption{Plot of the punctual absolute errors for $\theta=0.3$.}
    \end{subfigure}

    \vspace{0.5cm}

    \begin{subfigure}{0.45\textwidth}
        \centering
        \includegraphics[width=\linewidth]{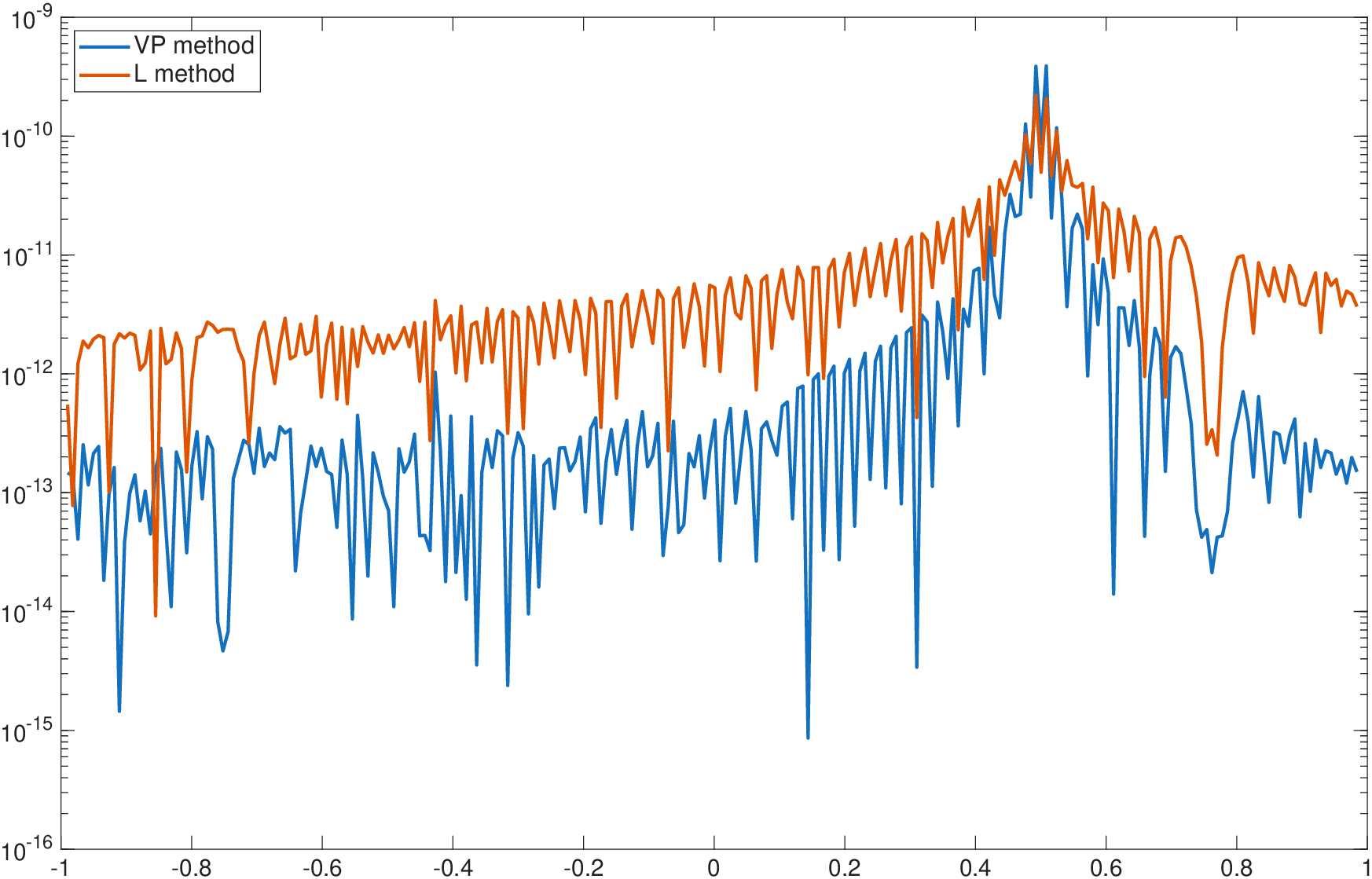}
        \caption{Plot of the punctual absolute errors for $\theta=0.5$.}
    \end{subfigure}
    \hfill
    \begin{subfigure}{0.45\textwidth}
        \centering
        \includegraphics[width=\linewidth]{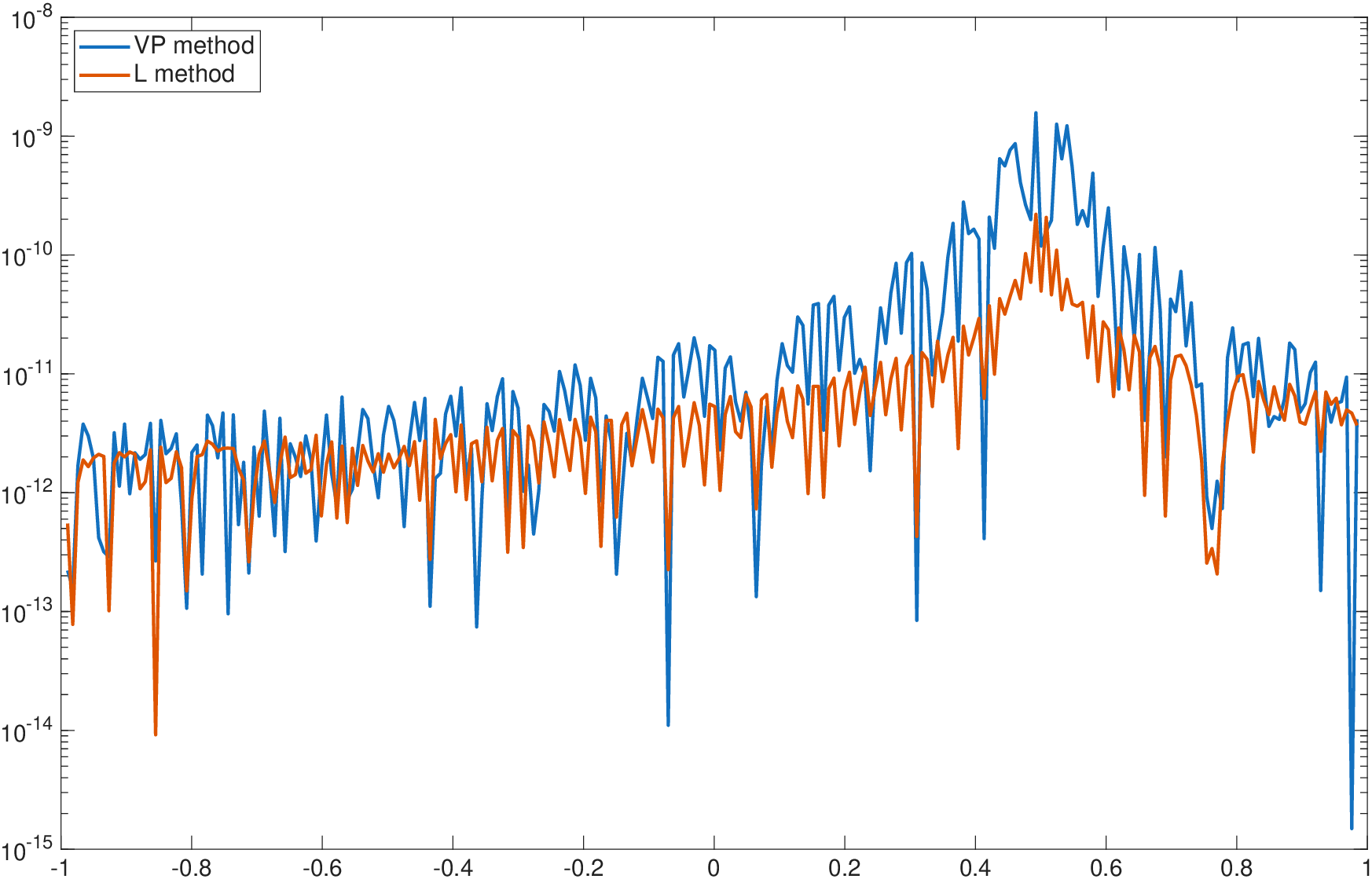}
        \caption{Plot of the punctual absolute errors for $\theta=0.8$.}
    \end{subfigure}

    \caption{Numerical results for Example \ref{ex8}.}
    \label{fig:ex8}
\end{figure}

\begin{table}[htbp]
    \centering
    \begin{tabular}{c|ccc|ccc}
    $n$ & $\hat\varepsilon_n^{VP}(f)$ & $\overline{\varepsilon}_n^{VP}(f)$ & $\kappa_\infty(\bm{N}_n)$  & $\hat\varepsilon_n^{L}(f)$ & $\overline{\varepsilon}_n^{L}(f)$ & $\kappa_\infty(\bm{T}_n)$  \\ \hline
    16   & 2.77e-04 & 7.00e-05 & 1.60  & 6.23e-05 & 1.61e-05 & 1.60 \\
    32   & 3.84e-06 & 3.59e-07 & 1.61  & 4.33e-06 & 9.54e-07 & 1.61  \\
    64   & 6.89e-07 & 3.57e-08 & 1.61  & 5.14e-07 & 4.51e-08 & 1.61  \\
    128  & 2.10e-08 & 8.45e-10 & 1.62  & 2.88e-08 & 3.41e-09 & 1.62  \\
    256  & 5.60e-09 & 7.83e-11 & 1.62  & 4.35e-09 & 1.16e-10 & 1.62  \\
    512  & 1.17e-10 & 1.85e-12 & 1.62  & 2.20e-10 & 9.57e-12 & 1.62      \\
    \end{tabular}
    \caption{Example \ref{ex8}: Maximum absolute errors comparison between the $VP$-method with $\theta^*=0.3$ and the $\mathcal{L}$-method.}
	\label{tab8}
\end{table}
\end{example}

\begin{example}\label{ex4}
	Let us consider the following equation
\begin{equation}\label{eq_ex4}
	f(y)-\frac{1}{\pi}\int_{-1}^{1} \frac{\sin{25 x}}{\sqrt{1-x^2}} f(x) dx = y- J_1(25)
\end{equation}
where $J_\nu(z)$ denotes the Bessel function of the first kind and the exact solution is $f(y)=y$. In view of Theorem  \ref{alternativa_fredholm}, and  being the right-hand side $g\in Z_r, \forall r>0$ with respect to the weight $u=v^{0,0}$, the solution belongs to $Z_r, \forall r>0$. The assumptions ensuring the convergence of the $VP$-method and the $\mathcal{L}$-method are both satisfied. Indeed, Table~\ref{tab4} highlights that both methods converge very rapidly to the exact solution independently from the choice of $\theta \in \Theta$ and the associated linear systems are well conditioned.

\begin{table}[htp!]
	\centering
	\begin{tabular}{c|cc|cc|cc|cc}
        & \multicolumn{2}{c|}{$\theta=0.2$}  & \multicolumn{2}{c|}{$\theta=0.5$} & \multicolumn{2}{c|}{$\theta=0.8$} & \multicolumn{2}{c}{$\mathcal{L}$-method} \\ \hline
        $n$ & $\epsilon_n^{VP}(f)$ & $\kappa_\infty(\bm{M}_n)$ & $\epsilon_n^{VP}(f)$ & $\kappa_\infty(\bm{N}_n)$ & $\epsilon_n^{VP}(f)$ & $\kappa_\infty(\bm{N}_n)$ & $\epsilon_n^{L}(f)$ & $\kappa_\infty(\bm{T}_n)$  \\ \hline
        4  & 7.77e-16 & 1.46 &  1.33e-15 & 1.33 &  2.18e-01 & 1.26 & 7.77e-16 & 1.46 \\
        8  & 4.66e-15 & 1.29 & 5.22e-15 & 1.31 & 6.22e-15 & 1.35 & 4.77e-15 & 1.36 \\
	\end{tabular}
	\caption{Example \ref{ex4}: Maximum absolute errors comparison between the $VP$-method with $\theta\in\{0.2,0.5,0.8\}$ and the $\mathcal{L}$-method}
	\label{tab4}
\end{table}
\end{example}

\begin{example}\label{ex5}
Let us consider the following equation
	\begin{equation}\label{eq_ex5}
		f(y)-\frac{1}{50 \pi}\int_{-1}^{1} f(x) \cos{(1500x)} (1-x)^{\frac 2 3}(1+x)^{-\frac 2 3}dx = \left|y-\frac 1 3\right|^{\frac{8}{3}},
	\end{equation}
whose exact solution is unknown. The right-hand side $g$ belongs to the Zygmund space $Z_{\frac{8}{3}}(u)$ for any weight $u$. Here, we choose $\gamma=\frac{3}{4}$ and $\delta=\frac{1}{4}$ so that the assumption assuring the convergence of both methods are satisfied. Consequently, by Theorem~\ref{convergence} and its corollary, the errors behaves as $\mathcal{O}(n^{-\frac{8}{3}})$ in the $VP$-method and as $\mathcal{O}\left(\log n \ n^{-\frac{8}{3}}\right)$ for the $\mathcal{L}$-method.
These predictions are confirmed in Table~\ref{tab5} by the maximum absolute errors. Figure~\ref{fig:ex5}(a) again shows the superior performance of the VP-method for $n=512$ with the optimal choice $\theta^*=0.3$, whereas Figure~\ref{fig:ex5}(b) illustrates how the maximum absolute error varies with $\theta\in\Theta$. As $\theta$ approaches $1$, the performance deteriorates because the operator $V_n^m(v^{\alpha,\beta})$ tends to the Fej\'er one \cite{themi2011}.
        
\begin{table}[htp!]
	\centering
	\begin{tabular}{c|ccc|ccc}
        $n$ & $\hat\varepsilon_n^{VP}(f)$ & $\overline{\varepsilon}_n^{VP}(f)$ & $\kappa_\infty(\bm{N}_n)$  & $\hat\varepsilon_n^L(f)$ & $\overline{\varepsilon}_n^L(f)$ & $\kappa_\infty(\bm{T}_n)$ \\ \hline
        8    & 3.85e-03 & 1.08e-03 & 1.03 & 3.43e-03 & 9.77e-04 & 1.02 \\
        16   & 4.91e-04 & 8.18e-05 & 1.05  & 6.08e-04 & 8.64e-05 & 1.03 \\
        32   & 4.59e-05 & 4.29e-06 & 1.08 & 7.69e-05 & 8.90e-06 & 1.05 \\
        64   & 1.85e-05 & 7.66e-07 & 1.10 & 1.13e-05 & 9.91e-07 & 1.06 \\
        128  & 2.90e-06 & 6.16e-08 & 1.08 & 1.72e-06 & 8.46e-08 & 1.08 \\
        256  & 3.23e-07 & 4.27e-09 & 1.09 & 3.60e-07 & 4.93e-09 & 1.10 \\
        512  & 1.81e-08 & 2.12e-10 & 1.12  & 3.70e-08 & 7.31e-10 & 1.12 \\
    \end{tabular}
	\caption{Example \ref{ex5}: Maximum and mean absolute errors comparison between the $VP$-method with $\theta^*=0.3$ and the $\mathcal{L}$-method.}
	\label{tab5}
\end{table}

\begin{figure}[htbp]
   \begin{subfigure}{0.45\textwidth}
        \centering
        \includegraphics[width=\linewidth]{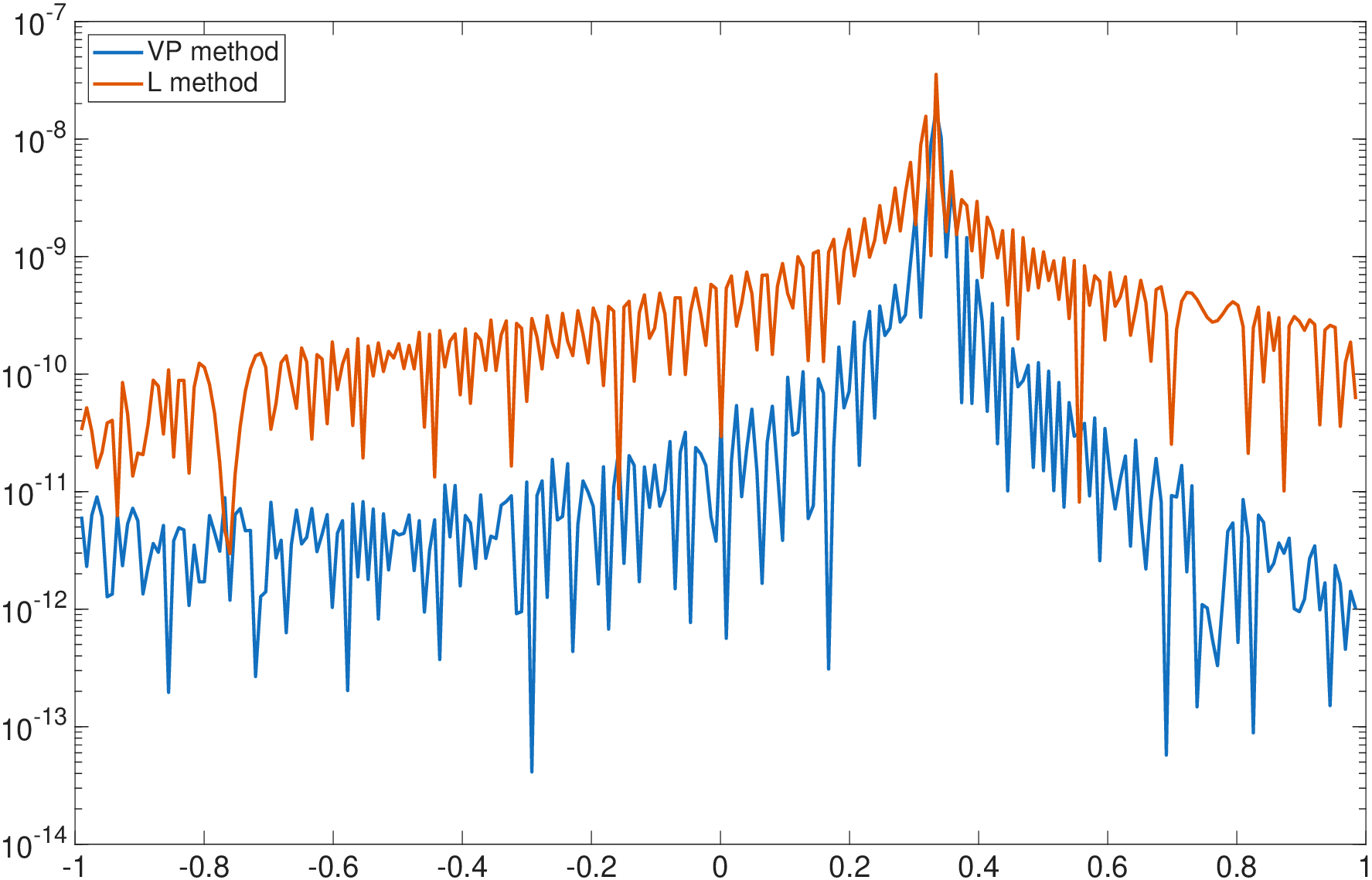}
        \caption{Plot of the punctual absolute error for the $VP$-method with $\theta^*=0.3$ and the $\mathcal{L}$-method.}
    \end{subfigure}
    \hfill
    \begin{subfigure}{0.45\textwidth}
        \centering
        \includegraphics[width=\linewidth]{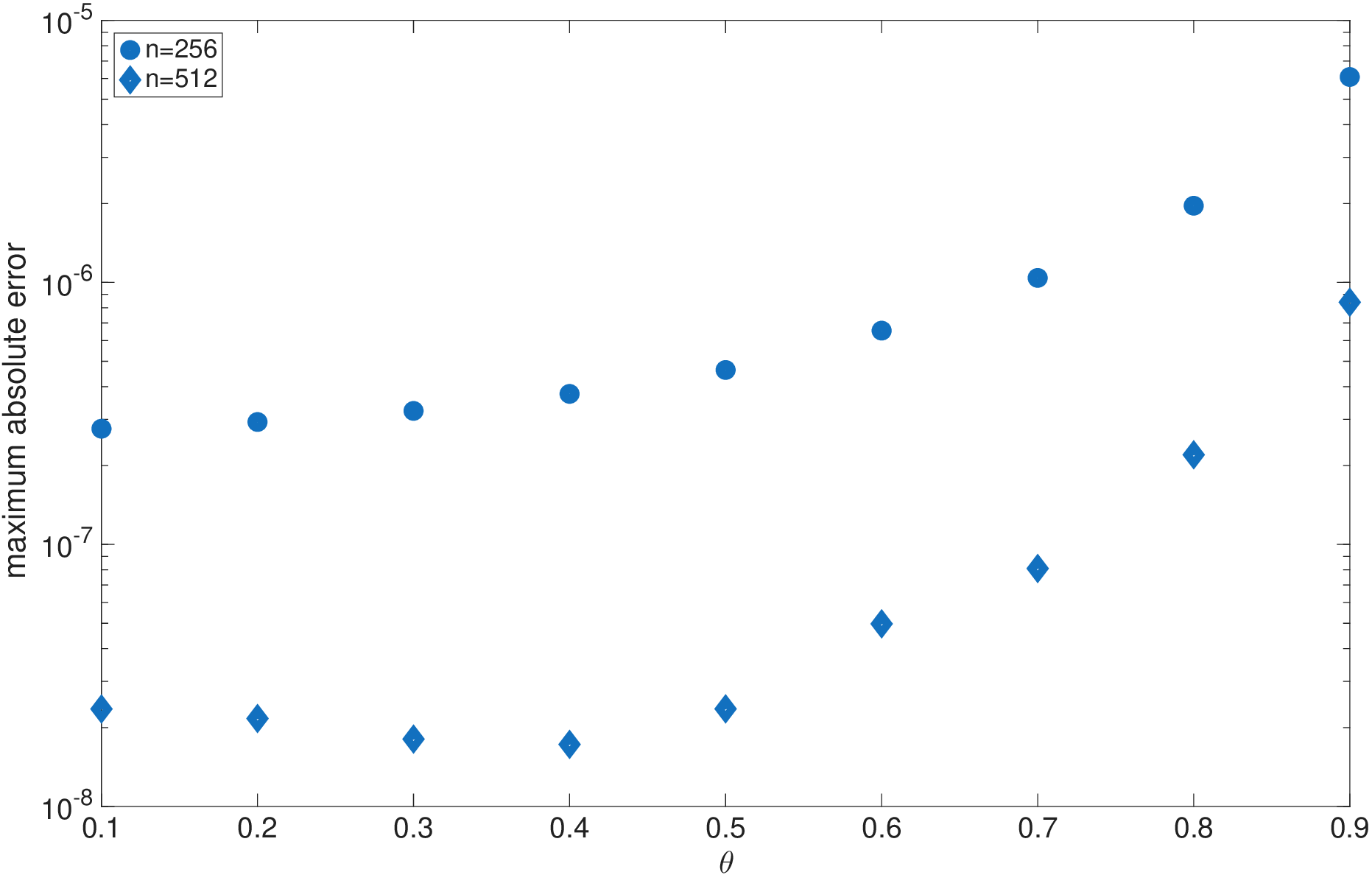}
        \caption{Benchmark analysis of the maximum absolute error behavior for different values of $\theta \in \Theta$ and fixed $n$.}
    \end{subfigure}

    \caption{Numerical results of Example \ref{ex5}.}
    \label{fig:ex5}
\end{figure}
\end{example}

\begin{example}\label{ex6}
	Let us consider the following equation
	\begin{equation}\label{eq_ese6}
		f(y)-\frac{9}{8}\int_{-1}^{1} f(x) \log{|x-y|} dx = g(y),
	\end{equation}
	where $g(y)=y^2+\frac1 8 (2+6y^2-3\log({1-y})+3y^3 \log({1-y})-3 \log({1+y})-3y^3\log({1+y}))$ and the exact solution is $f(y)=y^2$. Fixing  $u=v^{0,0}$, in view of Lemmas \ref{boundedness} and \ref{compattezza_zyg}, the solution $f\in Z_r(u)$ with $0<r<1$. However, in this context the theoretical assumptions of the $\mathcal{L}$-method are not satisfied. Table~\ref{tab6} shows that convergence of the $VP$-method with $\alpha=\beta=-\frac 1 2$ is achieved for the three choices $\theta\in \{0.1, \, 0.4,\, 0.8 \}$, although the case $\theta=0.8$ requires slightly more nodes to reach machine precision.
	
	\begin{table}[htp!]
		\centering
		\begin{tabular}{c|cc|cc|cc}
			& \multicolumn{2}{c|}{$\theta=0.1$} &\multicolumn{2}{c|}{$\theta=0.4$} & \multicolumn{2}{c}{$\theta=0.8$} \\ \hline
			$n$ & $\hat\varepsilon_n^{VP}(f)$ &  $\kappa_\infty(\bm{N}_n)$ & $\hat\varepsilon_n^{VP}(f)$ &  $\kappa_\infty(\bm{N}_n)$ & $\hat\varepsilon_n^{VP}(f)$ &  $\kappa_\infty(\bm{N}_n)$ \\ \hline
			4  & 2.00e-15 & 2.59  & 1.11e-15 & 2.50  & 2.16e-01 & 2.30   \\
			8  & 4.55e-15 & 3.79  & 4.55e-15 & 3.72  & 5.43e-02 & 3.67 \\
			16 & 7.22e-15 & 4.57  & 7.99e-15 & 4.55  & 1.11e-14 & 4.52   \\
		\end{tabular}
		\caption{Example \ref{ex6}: Maximum absolute errors of the $VP$-method with $\theta\in \{0.1,0.4,0.8\}$.}
		\label{tab6}
	\end{table}
\end{example}

\begin{example}\label{ex7}
Let us consider the following equation
	\begin{equation}\label{eq_ex7}
		f(y)-\frac{1}{8}\int_{-1}^{1} f(x) \log|x-y| dx = (1+y)^{\frac 3 4}|y|^{\frac{10}{3}},
	\end{equation}
whose exact solution is unknown. With $\gamma=\delta=0$, in view of Lemmas \ref{boundedness} and \ref{compattezza_zyg}, the solution belongs to the space 
 $Z_{r}, \ 0<r<1$. 
With this choice of $u$ and being $w=v^{0,0}$, the theoretical assumptions of the $\mathcal{L}$-method are not satisfied. Table~\ref{tab7} reports the results obtained by the VP-method for the two best choices of the parameter $\theta\in\Theta$. The errors confirm the theoretical expectations, while the conditioning results indicate the stability of the method.
Figure~\ref{fig:ex7} highlights that, at present, no theoretical criterion is available to guide the selection of the optimal value of $\theta$. Overall, the optimal choice depends on the considered example, and values closer to $1$ tend to reduce the performance of the $VP$-method.

\begin{table}[htp!]
	\centering
	\begin{tabular}{c|ccc|ccc}
		& \multicolumn{3}{c|}{$\theta=0.1$}
		& \multicolumn{3}{c}{$\theta=0.5$} \\ \hline
		$n$
		& $\hat\varepsilon_n^{VP}(f)$
		& $\overline{\varepsilon}_n^{VP}(f)$
		& $\kappa_\infty(\bm{N}_n)$
		& $\hat\varepsilon_n^{VP}(f)$
		& $\overline{\varepsilon}_n^{VP}(f)$
		& $\kappa_\infty(\bm{N}_n)$ \\ \hline
		
		8    & 5.34e-03 & 1.42e-03 & 1.36 & 1.50e-02 & 4.85e-03 & 1.35 \\
		16   & 1.65e-03 & 1.26e-04 & 1.41 & 2.69e-03 & 1.86e-04 & 1.40 \\
		32   & 2.20e-04 & 1.05e-05 & 1.42 & 1.92e-04 & 1.30e-05 & 1.42 \\
		64   & 3.88e-05 & 8.56e-07 & 1.44 & 4.43e-05 & 1.12e-06 & 1.44 \\
		128  & 3.47e-06 & 6.13e-08 & 1.47 & 5.12e-06 & 9.14e-08 & 1.47 \\
		256  & 1.41e-07 & 3.62e-09 & 1.48 & 2.51e-07 & 7.18e-09 & 1.48 \\
		512  & 1.87e-08 & 3.26e-10 & 1.49 & 2.38e-08 & 6.11e-10 & 1.49 \\
		
	\end{tabular}
	\caption{Example \ref{ex7}: Maximum and mean absolute errors of the $VP$-method with $\theta=0.1$ and $\theta=0.5$.}
	\label{tab7}
\end{table}

\begin{figure}[htbp]
   \begin{subfigure}{0.45\textwidth}
        \centering
        \includegraphics[width=\linewidth]{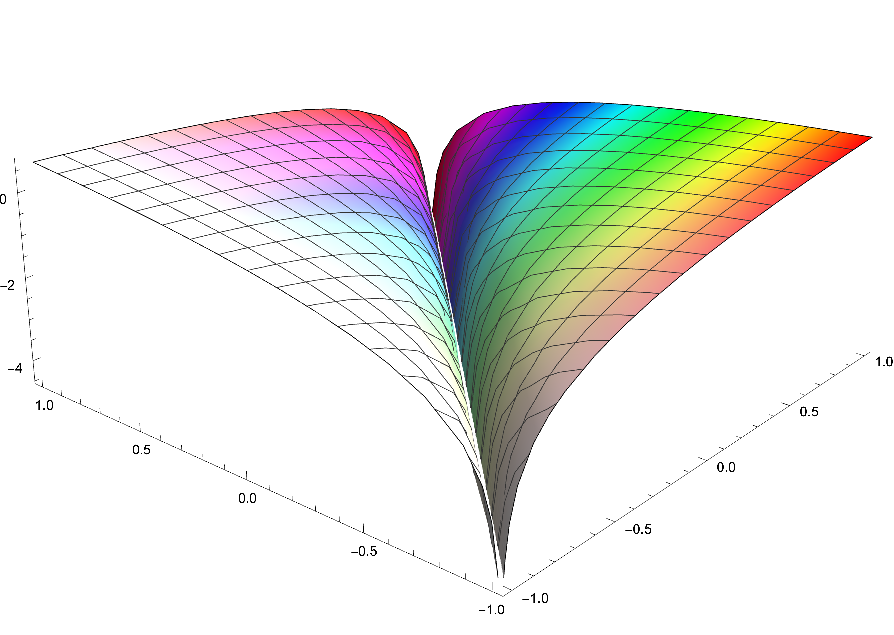}
        \caption{Plot of the logarithmic kernel $k(x,y)=\log |x-y|$.}
    \end{subfigure}
    \hfill
    \begin{subfigure}{0.45\textwidth}
        \centering
        \includegraphics[width=\linewidth]{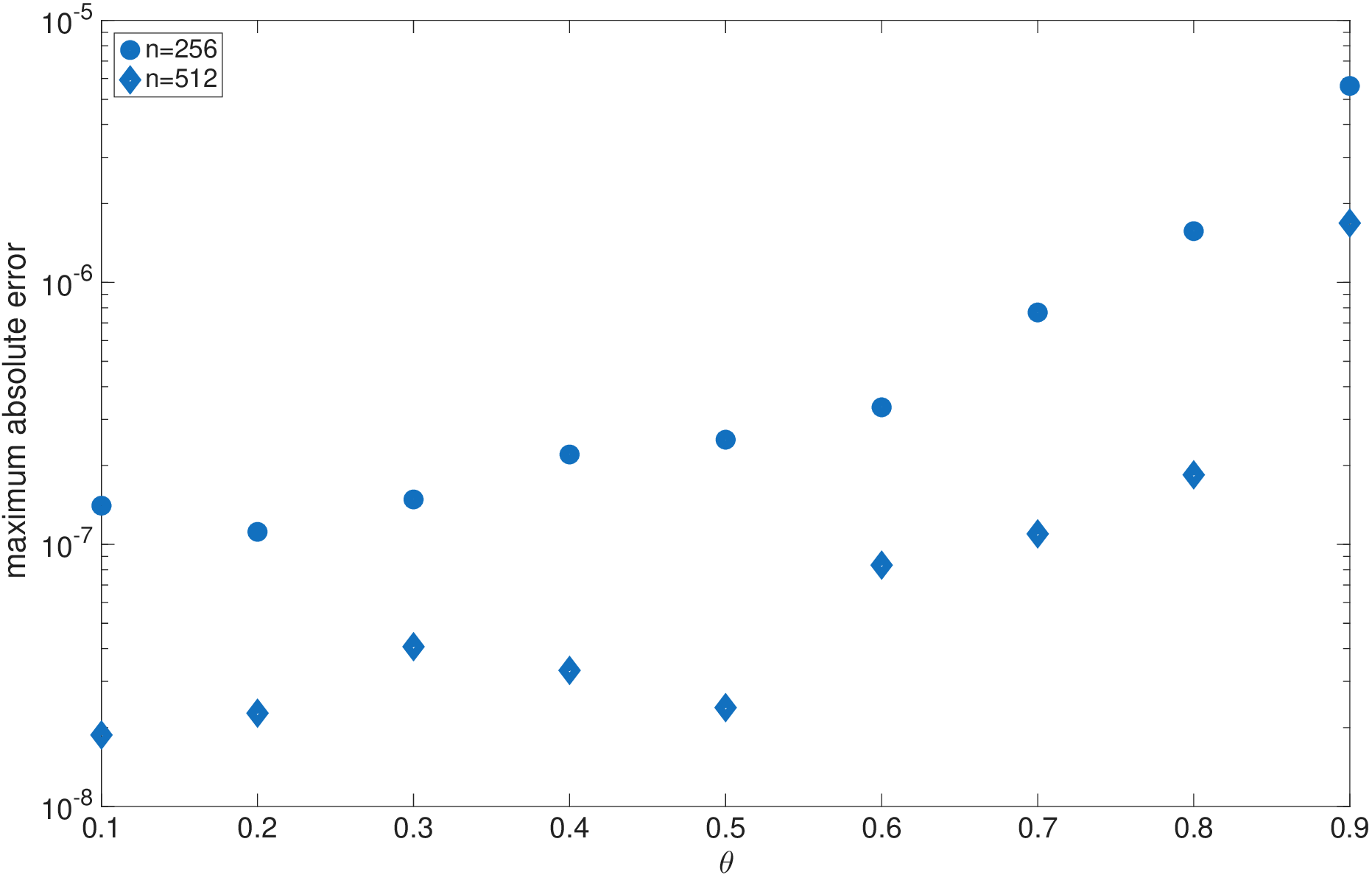}
        \caption{Benchmark analysis of the maximum absolute error behavior for different values of $\theta \in \Theta$ and fixed $n$.}
    \end{subfigure}

    \caption{Numerical results for Example \ref{ex7}.}
    \label{fig:ex7}
\end{figure}

\end{example}

\section{The proofs}\label{sec:proofs}

\begin{proof}[Proof of Lemma \ref{boundedness}]
Let us first prove that $K^\mu:C_u\to C_u$ is  a bounded operator. 

For $\mu\neq 0$, start from
$$u(y)|K^\mu f(y)| \le u(y) \|f\|_{C_u}\int_{-1}^1 |x-y|^\mu \frac{w(x)}{u(x)}dx=:\|f\|_{C_u} T_\mu(y).$$

If $\mu>0$  and $w/u\in L^1$, then the statement follows from $$ T_\mu(y)\le \C \int_{-1}^1 \frac{w(x)}{u(x)}dx\le \C\ne \C(y).$$
Consider now the case $-1<\mu<0$.
If  $|y|<1$, the boundedness of $T_\mu$ follows  taking into account that $\frac{w(x)(1-x^2)^{\mu}}{u(x)}\in L^1$  implies $\frac{w(x)}{u(x)}\in L^1$, and recalling the following lemma

\noindent \textit{Lemma A\cite{CapobiancoJLuthermastro}: 
For $-1<\mu<0$, \ $0\le \tau,\nu<1$, and $-1<y<1$, it is}
$$\int_{-1}^1 |x-y|^\mu v^{-\tau,-\nu}(x)dx\le 
[\mathrm{B}(1 - \nu, 1 +\mu)(1+ y)^{1+\mu} + \mathrm{B}(1 - \tau, 1 +\mu)(1- y)^{1+\mu}] v^{-\tau,-\nu}(y),$$
\textit{where $\mathrm{B}(\alpha, \beta) = \int_{0}^1 
t^{\alpha-1}(1 - t)^{\beta-1} dt, \alpha,\beta> 0,$ denotes the Beta function.}

For $y=\pm 1,$ the assumption $w(x)(1-x^2)^\mu\in L^1$ is necessary to assure  the boundedness of  $T_\mu(\pm 1)$.

Finally, in the case $\mu=0$, the boundedness of $K^\mu:C_u\to C_u$ follows  as in the previous case, recalling that 
$\int_0^a t^\eta \log t dt<\infty$, for any $\eta>-1$, and 
taking into account that 
$$|\log|x-y||\le |x-y|^{-\varepsilon},\ \forall \varepsilon>0.$$

Now we prove that $K^\mu:C_u\to C_u$ is a compact operator.
In view of the boundedness of the linear operator $K^\mu:C_u\to C_u$,  we have to prove that the set
    $$\mathcal{F}=\{  K^\mu f: \|f\|_{C_u}\le 1\}$$ is  relatively compact  in $C_u$. 
To this end, since $C_u$ is a compact metric space, in view of the Ascoli-Arzel\`a theorem, we have to prove that $F$ is a uniformly bounded and a equicontinuous family, i.e.
\begin{equation}\label{unif_bound}
    \sup_{|y|\le 1} u(y)\int_{-1}^1 |k^\mu(x,y)| \frac{w(x)}{u(x)}<\infty,
\end{equation}
\begin{equation}\label{equicont}
    \lim_{y\to y_1}|K_\mu f(y) u(y)-K_\mu f(y_1) u(y_1)|=0, \quad \forall y_1\in [-1,1].
\end{equation}
The bound \eqref{unif_bound} holds in view of the boundedness of $K^\mu$. To prove \eqref{equicont}, for $f: \|f\|_{C_u}\le 1$, we consider
   
\begin{eqnarray*}
    |K_\mu f(y) u(y)-K_\mu f(y_1) u(y_1)| &= &\left | \int_{-1}^1 [u(y)k^\mu(x,y)-u(y_1)k^\mu(x,y_1)]f(x) w(x)dx\right| \\
    &\le &|u(y)-u(y_1)|\int_{-1}^1 |k^\mu(x,y)|\frac{w(x)}{u(x)} dx  \\ 
    &+& u(y_1)\int_{-1}^1 |k^\mu(x,y)-k^\mu(x,y_1)|\frac{w(x)}{u(x)} dx.
\end{eqnarray*}
The first term tends to $0$ in view of the continuity of $u$ and of the boundedness of $\int_{-1}^1 |k^\mu(x,y)|\frac{w(x)}{u(x)} dx$ that we have already proved  for any $y\in [-1,1]$.

The second term also tends to $0$ as $y\to y_1$ in view of \cite[Lemma 6.3.10]{JMN} for $\mu=0$ and \cite[Lemma 6.3.11]{JMN} for $\mu\neq 0$. Hence,  \eqref{equicont} holds, and the lemma is completely proved.
\end{proof}

\vspace{0.3cm}

\begin{proof}[Proof of Lemma \ref{compattezza_zyg}]

Assuming $y\in [-1+(2hk)^2,1-(2hk)^2]$ and for $0<h\le t$, we get 
\begin{eqnarray}\nonumber| \Delta_{h\varphi(y)}^k \left(K^\mu f\right)(y) u(y)| &\le & 
\left| \int_{-1}^1 \Delta_{h\varphi(y)}^k \left(k^\mu _x(y)\right)f(x) w(x) dx\right|u(y)\le  \|f\|_{C_u} \int_{-1}^1 u(y)\left| \Delta_{h\varphi(y)}^k \left(k^\mu_x(y)\right)\right|\frac{ w(x)}{u(x)} dx\\ \nonumber
&= & \|f\|_{C_u}\left\{\int_{-1}^{y-2kh\varphi(y)}+\int_{y-2kh\varphi(y)}^{y+2kh\varphi(y)}+\int_{y+2kh\varphi(y)}^1\right\} u(y)\left| \Delta_{h\varphi(y)}^k \left(k^\mu_x(y)\right)\right|\frac{ w(x)}{u(x)} dx\\
& =: & \|f\|_{C_u}\left\{ D_1(y,h)+D_2(y,h)+D_3(y,h)\right\}.\label{somma_Di}
\end{eqnarray}

Firstly, we estimate $D_2(y,h)$. From the definition of $\Delta_{h\varphi(y)}^k\left(k^\mu_x(y)\right)$ we have 
\begin{eqnarray*}
    D_2(y,h)&=&\int_{y-2kh\varphi(y)}^{y+2kh\varphi(y)}u(y)\left| \Delta_{h\varphi(y)}^k \left(k^\mu_x(y)\right)\right|\frac{ w(x)}{u(x)} dx \\ 
    &\le & \C u(y)\sum_{j=0}^k \binom k j \int_{y-2kh\varphi(y)}^{y+2kh\varphi(y)}\left|k^\mu_x\left(y+(k-2j)\frac h 2 \varphi(y)\right)\right| dx.
\end{eqnarray*}
    
Introducing the change of  variable $x=y+2hk\varphi(y)z$, we have
\begin{eqnarray}\label{D2}
    D_2(y,h)&\le &C u(y)(h\varphi(y))^{\mu+1} \sum_{j=0}^k \binom k j \int_{-1}^{1}\left|2zk-\frac k2+j\right|^{\mu} dz \le  \C h^{\mu+1}, \quad \mu\ne 0,  
\end{eqnarray}
and by \cite[proof of Lemma 4.1]{iwota}
\begin{equation}\label{D2zero}
    D_2(y,h)\le C h \log (h^{-1}),\quad \mu=0.
\end{equation}
   
Concerning the estimate of $D_3$ and $D_1$, we recall that, given a function  $F$, the following relation  holds for any $r\in \NN$ \cite[p.21, (2.4.5)]{DT}
\begin{equation}\label{f_dt}
    \Delta_{h\varphi(x)}^r F(x)=\int_{-\frac h 2 \varphi(x)}^{\frac h 2 \varphi(x)}\dots \int_{-\frac h 2 \varphi(x)}^{\frac h 2 \varphi(x)} F^{(r)}(x+t_1+t_2+\dots+t_r)d t_1 d t_2 \dots d t_r.
\end{equation}
This relation implies the following
\begin{equation}\label{delta}
    |\Delta_{h\varphi(x)}^r F(x)|\le \C (h\varphi(x))^{r-1}\int_{-r\frac h 2 \varphi(x)}^{r \frac h 2 \varphi(x)} |F^{(r)}(x+t)|dt,\qquad \forall r\in\NN.
\end{equation}
For instance, to  prove \eqref{delta} for $r=2$, we use the changes of variables   $t_{1}=\frac{t+v}2, t_2=\frac{t-v}2$ in \eqref{f_dt}, and we get 
\begin{eqnarray*}
   \Delta_{h\varphi(x)}^2 F(x)&=&\int_{-\frac h2 \varphi(x)}^{\frac h2 \varphi(x)} \int_{-\frac h2 \varphi(x)}^{\frac h2\varphi(x)} F^{''}(x+t_1+t_2)d t_1 d t_2 =\frac 12\int_{-h \varphi(x)}^{h \varphi(x)}dv \int_{-h \varphi(x)}^{h\varphi(x)} F^{''}(x+t)d t \\
   &=& h\varphi(x) \int_{-h \varphi(x)}^{h\varphi(x)} F^{''}(x+t)d t.
\end{eqnarray*}
The same reasoning can be repeated for $r>2$.
   
If we use  \eqref{delta} with  $F(y):=k^\mu_x(y)$ and $r=k$, where $k=1$ if $\mu=0$ and $k>\mu+1$ if $\mu\neq 0$, then we  get
\begin{eqnarray}| \nonumber \Delta_{h\varphi(y)}^k \left(k^\mu _x(y)\right)| &\le &  \C \left(\frac{hk}2\varphi(y)\right)^{k-1} \int_{-\frac{hk}2\varphi(y)}^{\frac{hk}2\varphi(y)}
\left| \frac{\partial^k }{\partial y^k}k^\mu(x,\ y+t)   \right|dt \\
\label{Delta}  & = & \C \left(\frac{hk}2\varphi(y)\right)^{k-1}\int_{-\frac{hk}2\varphi(y)}^{\frac{hk}2\varphi(y)} |x-y-t|^{\mu-k}dt. 
\end{eqnarray}

Now, we estimate $D_3$. In view of the assumptions on $x,y$ in $D_3$ it is $x-y-t>0$ and using $w(x)/u(x)\le \C$, for any  $k>\mu+1$ with $\mu\neq 0$ we get
\begin{eqnarray}
    D_3(y,h) &\le & \C u(y) \left(\frac{hk}2\varphi(y)\right)^{k-1}  \int_{-\frac{hk}2\varphi(y)}^{\frac{hk}2\varphi(y)} \left(\int_{y+2hk\varphi(y)}^1 (x-y-t)^{\mu-k}dx\right)dt\label{d31}\\ \nonumber
    &\le & \C u(y) \left(\frac{hk}2\varphi(y)\right)^{k-1}
    \int_{-\frac{hk}2\varphi(y)}^{\frac{hk}2\varphi(y)}
    \left( (2hk\varphi(y)-t)^{\mu-k+1}- ( 1-y-t)^{\mu-k+1}\right) dt 
    \\ &\le & \C u(y) \left(\frac{hk}2\varphi(y)\right)^{k-1}\left(  \frac 3 2hk\varphi(y)\right)^{\mu-k+1}
     \int_{-\frac{hk}2\varphi(y)}^{\frac{hk}2\varphi(y)}
    dt \le \C u(y)(hk\varphi(y))^{\mu+1}\nonumber
\end{eqnarray}
and hence 
\begin{equation}\label{D3} 
    D_3(y,h) \le  \C h^{\mu+1},\quad \mu\neq 0.
\end{equation}
   
Consider now  $\mu=0$, and $k=1$  by \eqref{Delta} we get
\begin{eqnarray}\nonumber
    D_3(y,h) &\le & \C u(y) \left(\frac{hk}2\varphi(y)\right)^{k-1}  \int_{-\frac{hk}2\varphi(y)}^{\frac{hk}2\varphi(y)} \left(\int_{y+2hk\varphi(y)}^1 (x-y-t)^{-1}dx\right)dt\\ \nonumber &\le & \C  \int_{-\frac{hk}2\varphi(y)}^{\frac{hk}2\varphi(y)}\bigg( \log(1-y-t)-\log(2hk\varphi(y)-t)\bigg) dt\le \C  \int_{-\frac{hk}2\varphi(y)}^{\frac{hk}2\varphi(y)}\bigg( -\log(2hk\varphi(y)-t)\bigg) dt\\ &= & \C  \int_{-\frac{hk}2\varphi(y)}^{\frac{hk}2\varphi(y)}\log\bigg(\frac 1 {2hk\varphi(y)-t}\bigg) dt\le \C h \log h^{-1}\label{D3zero}
\end{eqnarray}

To estimate $D_1(y,h)$ we proceed in the same way as for $D_3$. For $\mu\neq 0$,  by \eqref{Delta}, taking into account that $t+y-x>0$ and for  $k>\mu+1$ we get
\begin{eqnarray*}
    D_1(y,h) &\le & \C u(y) \left(\frac{hk}2\varphi(y)\right)^{k-1}  \int_{-\frac{hk}2\varphi(y)}^{\frac{hk}2\varphi(y)} \left(\int_{-1}^{y-2hk\varphi(y)}(t+y-x)^{\mu-k}dx\right)dt\\
    &\le & \C u(y) \left(\frac{hk}2\varphi(y)\right)^{k-1}\int_{-\frac{hk}2\varphi(y)}^{\frac{hk}2\varphi(y)}\left( t+2hk\varphi(y)\right)^{\mu-k+1}dt  \\ 
    &\le & \C u(y) \left(\frac{hk}2\varphi(y)\right)^{k-1} \left(  \frac 3 2hk\varphi(y)\right)^{\mu-k+1}\int_{-\frac{hk}2\varphi(y)}^{\frac{hk}2\varphi(y)}  dt , 
\end{eqnarray*}
and hence
 \begin{equation}\label{D1} D_1(y,h) \le  \C h^{\mu+1}, \quad \mu\neq 0.\end{equation}
Similarly, in the case $\mu=0, \ k=1$ we have
\begin{equation}\label{D1zero}
    D_1(h,y)\le \C h\log (h^{-1}).
\end{equation}

Combining \eqref{D1},\eqref{D2}, \eqref{D3} with \eqref{somma_Di} for $\mu\neq 0$ and any $k>\mu +1$, we can conclude
$$|\Delta_{h\varphi(y)}^k \left(K ^\mu f\right)(y) u(y)| \le \C \|f\|_{C_u} h^{\mu+1},\ \ $$
and hence 
\begin{equation}\label{omega_ipo_1}
    \sup_{t>0}\frac{\Omega^k_\varphi(K^\mu f,t)}{t^{\mu+1}}\le \C \|f\|_{C_u},\quad \C\neq \C(f),\quad \mu\neq 0.
\end{equation}
Combining \eqref{D1zero},\eqref{D2zero}, \eqref{D3zero} with \eqref{somma_Di}
for $\mu=0$ and $k=1$ we have
$$|\Delta_{h\varphi(y)}\left(K ^\mu f\right)(y) u(y)| \le \C \|f\|_{C_u} h \log (h^{-1}).\ \ $$
Consequently, it follows
\begin{equation}\label{omega_ipo_2}
    \sup_{t>0}\frac{\Omega_\varphi(K^\mu f,t)}{t \log (t^{-1})}\le \C \|f\|_{C_u},\quad \C\neq \C(f), \quad \mu=0.
\end{equation}

In conclusion, by  \eqref{omega_ipo_1}--\eqref{omega_ipo_2}, recalling also \cite[Lemma 3.3]{JL}, we have proved that $K^\mu:C_u\to Z_r(u)$ is bounded for all $0<r\le \mu+1$ if $\mu\ne 0$ and for all $0<r<1$ if $\mu=0$. 

Therefore, the compactness of $K^\mu: Z_r(u)\to Z_r(u)$ follows from this result and the compact embedding $Z_r(u)\subset C_u$. 
 \end{proof}

\vspace{0.3cm}

\begin{proof}[Proof of Theorem \ref{prop-dual}]
First of all, let us define the following polynomials
\begin{equation}\label{fi-dual}
\tilde\phi_{n,h}^m(x)=\sum_{i=0}^{n-1}\frac 1{\mu_{n,i}^m} p_i(v^{\alpha,\beta}, x_h)p_i(v^{\alpha,\beta},x),\qquad h=1,\ldots,n, \qquad |x|\le 1 
\end{equation}
and, set
\[
<f,g>:=\int_{-1}^1 f(x)g(x)v^{\alpha,\beta}(x) dx, 
\]
let us prove that
\begin{equation}\label{orto}
<\phi_{n,k}^m, \ \tilde\phi_{n,h}^m >=\delta_{h,k}=\left\{\begin{array}{ll}
  1   & h=k \\
  0   & h\ne k
\end{array}\right.  \qquad h,k=1,\ldots,n 
\end{equation}
Indeed, for any $h,k=1,\ldots,n$, we have
\begin{eqnarray*}
<\phi_{n,k}^m, \ \tilde\phi_{n,h}^m > &=&
\lambda_k \sum_{j=0}^{n+m-1}\sum_{i=0}^{n-1}\frac{\mu_{n,j}^m}{\mu_{n,i}^m}p_j(v^{\alpha,\beta},x_k)p_i(v^{\alpha,\beta},x_h)\int_{-1}^1 p_j(v^{\alpha,\beta},x)p_i(v^{\alpha,\beta},x)v^{\alpha,\beta}(x)dx\\
&=& \lambda_k \sum_{j=0}^{n+m-1}\sum_{i=0}^{n-1}\frac{\mu_{n,j}^m}{\mu_{n,i}^m}p_j(v^{\alpha,\beta},x_k)p_i(v^{\alpha,\beta},x_h) \delta_{i,j}\\
&=& \lambda_k \sum_{i=0}^{n-1}p_i(v^{\alpha,\beta},x_k)p_i(v^{\alpha,\beta},x_h) =\delta_{h,k}
\end{eqnarray*}
where the last identity easily follows from \eqref{lambda} and the Christoffel-Darboux formula.

Now, let us assume there are some coefficients $a_k$, $k=1,\ldots,n$, such that
\[
\sum_{k=1}^n a_k \phi_{n,k}^m(x)=0, \qquad \forall x\in [-1,1]
\]
This implies that
\[
< \sum_{k=1}^n a_k \phi_{n,k}^m, \ \tilde\phi_{n,h}^m>=0, \qquad h=1,\ldots,n
\]
and using \eqref{orto}, we get
\[
0 = < \sum_{k=1}^n a_k \phi_{n,k}^m, \ \tilde\phi_{n,h}^m>= \sum_{k=1}^n a_k\delta_{h,k}=a_h, \qquad h=1,\ldots,n
\]
which concludes the proof.
\end{proof}

\vspace{0.3cm}

\begin{proof}[Proof of Theorem \ref{convergence}]
First, observe that by Theorem \ref{th-VPinf}, we have
\begin{equation}\label{lim-Knm}
    \lim_{n\to\infty} \|Kf-K_n^mf\|_{C_u} =0, \qquad \forall f\in C_u.
\end{equation}

In order to prove that $I-K_n^m:C_u\to C_u$ is invertible, we consider the following decomposition 
\begin{equation}\label{eq:tmp1}
		I-\nu K_n^m=(I-\nu K) [I-\nu (I-\nu K)^{-1}(K_n^m-K)]
\end{equation}
    and observe that  \eqref{lim-Knm} implies that for sufficiently large $n$ we have

	\begin{equation*}
		 \| \nu(I-K)^{-1}(K_n^m-K) \| <1.
	\end{equation*}
  which (see, e.g. \cite[Theorem A.1]{atkinson}) ensures that the operator $I-\nu (I-K)^{-1}(K_n^m-K)$ is invertible and  we have
     \[
     \|(I-\nu (I-K)^{-1}(K_n^m-K))^{-1}\|\le \frac 1{1- \| (I-K)^{-1}\|\|(K_n^m-K) \|.   }   \]
 Consequently, if $(I-\nu K)$ is invertible then, by \eqref{eq:tmp1},  $(I-\nu K_n^m)$ is also invertible and its inverse is uniformly bounded w.r.t. $n$.
 
This proves the existence of a unique solution $f_n^m$ of \eqref{eqdiscr}.   This solution belongs to $S_n^m$ since $f_n^m=g_n^m+\nu K_n^m f_n^m$ and both the addenda belong to $S_n^m.$

 To establish  \eqref{lim-cond}, we note that we can write
	 \begin{eqnarray*}
    (I-\nu K_n^m)-(I-\nu K)&=&  \nu(K-K_n^m)\\
	(I-\nu K)^{-1}-(I-\nu K_n^m)^{-1}&=&(I-\nu K)^{-1}(\nu K-\nu K_n^m)(I-\nu K_n^m)^{-1}
	 \end{eqnarray*}
Hence, using \eqref{lim-Knm}, we get
   {\small{  \begin{eqnarray*}
    \left|\|I-\nu K_n^m\|-\|I-\nu K\|\right|&\le &|\nu| \| K-K_n^m\|\to 0\\
	\left|\|(I-\nu K)^{-1}\|-\|(I-\nu K_n^m)^{-1}\|\right| &\le& |\nu|\|(I-\nu K)^{-1}\|\ \|K-K_n^m\|\ \|(I-\nu K_n^m)^{-1}\|\to 0
	 \end{eqnarray*}}}
     and \eqref{lim-cond} follows.

    Finally, estimate \eqref{errore_equazione} follows from the following identity
    \begin{equation}\label{identity}
       f^* - f_n^{m^*} = (I-\nu K_n^m)^{-1} [(g-g_n^m)+\nu(K-K_n^m)f^*],
    \end{equation}
    taking into account that $(I-\nu K_n^m)^{-1}$ is uniformly bounded and recalling Theorem \ref{th-VPinf}.
\end{proof}

\vspace{0.3cm}

\begin{proof}[Proof of Corollary \ref{cor1}]
We recall that for all $s>0$ the space $Z_s(u)$ is compactly embedded in $C_u$, and for any $r>s$ $Z_r(u)$ is compactly embedded in $Z_s(u)$ \cite{JL}.  Consequently, the hypothesis $K:C_u\to Z_r(u)$ bounded implies that the map $K:Z_s(u)\to Z_s(u)$ is compact for all $s\le r$. This fact and the Fredholm alternative theorem ensures that $(I-\nu K):Z_s(u)\to Z_s(u)$ is invertible, i.e. for all $g\in Z_s(u)$ there exists a unique solution $f^*\in Z_s(u)$. 

Finally, the estimate \eqref{eqcor1} follows from \eqref{errore_equazione} by using \eqref{stimaBAE} for  $g, Kf^*\in Z_s(u)$ and taking into account that
\[
\|Kf^*\|_{Z_s(u)}\le \C \|f^*\|_{C_u}= \C \|(I-\nu K)^{-1}g\|_{C_u}\le \C \|g\|_{C_u}\le\C \|g\|_{Z_s(u)},
\]
the thesis follows.
\end{proof}
\vspace{0.3cm}

\begin{proof}[Proof of Theorem \ref{stability}] 
	Let the discrete operator $(I-\nu K_n^m):S_n^m\to S_n^m$ be represented by the matrix $\bm{N}_n= (\bm{I}_n-\nu \bm{H}_n)$  in the basis $$\mathcal{B}=\left\{ b_i(y):=\frac{\Phi_n^m(y)}{u(x_i)},\ i=1,2,\dots,n\right\}$$ of $S_n^m$.
    This means that if $F,G\in S_n^m$
    are represented by the vectors 
    $\mathbf{F}=[F_1,F_2,\dots,F_n], \ \mathbf{G}=[G_1,G_2,\dots, G_n]$ in the basis $\mathcal{B}$, then 
    \begin{equation}\label{equivalenza}(I-\nu K_n^m)F=G \ \Leftrightarrow \ \bm{N}_n \bm F =\bm G.\end{equation}
    \noindent 
    Moreover, in view of Theorems \ref{th-Marci} and \ref{prop-Cheb},     $\mathcal{B}$
 is an interpolating  Marcinkiewicz basis, that is $\forall F\in S_n^m$ it is
 $F(y)=\sum_{i=1}^n F_i b_i(y),$ with $F_i=F(x_i)u(x_i)$
and, in addition, 
$$\|F\|_{C_u}\sim \max_{i=1,\dots,n} |F_i|=\|\bm F\|_{\infty}.$$
Assuming that 
$\bm N_n \bm F =\bm G$, by \eqref{equivalenza} it is
$$\|\bm {N}_n \bm F \|_\infty=\|\bm G\|_\infty \sim  \|G\|_{C_u} =
\|(I-\nu K_n^m)F \|_{C_u}$$
and hence 
$$\|\bm {N}_n\|_\infty= \sup_{\bm F \neq \bm 0} \frac{\|\bm {N}_n \bm F\|_\infty}{\|\bm F\|_\infty}\sim   \sup_{ F\in S_n^m, F \neq  0} \frac{\|(I-\nu K_n^m)F\|_{C_u}}{\|F\|_{C_u}},$$
that is 
$$\|\bm {N}_n\|_\infty\sim \|(I-\nu K_n^m)\|. $$
Similarly proceeding, 
\begin{eqnarray*}\|\bm {N}_n^{-1}\|_\infty&=& \sup_{\bm G \neq \bm 0} \frac{\|\bm {N}_n ^{-1}\bm G\|_\infty}{\|\bm G\|_\infty}=
\sup_{\bm G \neq \bm 0} \frac{\|\bm F\|_\infty}{\|\bm G\|_\infty}\\ 
&\sim &   \sup_{ G \neq  0} \frac{\|(I-\nu K_n^m)^{-1}G \|_{C_u}}{\|G\|_{C_u}} =  \|(I-\nu K_n^m)^{-1}\|.\end{eqnarray*}
In conclusion, we have proved that 
$$\|\bm N_n\|_\infty \|\bm N_n^{-1}\|_\infty\sim \kappa( I- \nu K_n^m ),$$
and recalling  \eqref{lim-cond},  \eqref{finale} follows.
\end{proof}

\vspace{0.3cm}

\begin{proof}[Proof of Proposition \ref{prop_transform}]
    Consider the Fourier expansion of the polynomial $p_k(v^{\alpha,\beta})$ in the orthonormal system $\left\{ p_j(w) \right\}_{j\ge 0}$ 
\begin{equation}\label{eq:Fourier_exp}
    p_k(v^{\alpha,\beta},x)=\sum_{j=0}^k c_{k,j}p_j(w,x), \quad k=0,\ldots,n+m-1,
\end{equation}
\begin{equation}\label{eq:Fourier_coeff}
    c_{k,j}=\int_{-1}^{1} p_k(v^{\alpha,\beta},x)p_j(w,x)w(x) \, dx
\end{equation}
with $c_{k,j}=0, \; \forall\,j>k$.
Denoting by  $\left\{x_{n+m,h}(w)\right\}_{h=1}^{n+m}$ the zeros of the polynomial $p_{n+m}(w)$ and by $\left\{\lambda_{n+m,h}\right\}_{h=1}^{n+m}$ the Christoffel numbers related to $w$, 
the coefficients in \eqref{eq:Fourier_coeff} can be exactly computed by the $(n+m)$ Gauss-Jacobi quadrature rule, i.e. 
\begin{equation}\label{eq:Fourier_Gauss}
    c_{k,j}=\sum_{h=1}^{n+m} \lambda_{n+m,h}(w)p_k(v^{\alpha,\beta},x_{n+m,h}(w))p_k(w,x_{n+m,h}(w)). 
\end{equation}

In view of \eqref{eq:Fourier_Gauss}, we obtain the lower triangular transformation matrix of order $n+m$ 
\begin{equation*}
    \bm{C} = \bm{P}_v \bm{\Lambda}_{n+m} \bm{P}_w,
\end{equation*}
and hence, setting  
$$\bm{p(v^{\alpha,\beta})}_{\mathbf{n+m}}=[ p_0(v^{\alpha,\beta}),p_1(v^{\alpha,\beta}),\dots,p_{n+m-1}(v^{\alpha,\beta})]^T$$
 $$\bm{p(w)}_{\mathbf{n+m}}=[ p_0(w),p_1(w),\dots,p_{n+m-1}(w)]^T$$
it is 
$$\bm{p(v^{\alpha,\beta})}_{\mathbf{n+m}}=\bm{C}\bm{p(w)}_{\mathbf{n+m}},$$
and hence \eqref{matrix_moment_transform} easily follows.
\end{proof}

\vspace{0.3cm}

\begin{proof}[Proof of Proposition \ref{thm:Momenti_log}]
To prove \eqref{eq:Ricorrenza_Momenti_Logaritmici}, start by recalling that, according to Rodrigues' formula \cite[Section 4.10]{Szego}, the identity
\begin{equation}\label{eq:Rodriguez}
    p_k (v^{\rho,\sigma},x) \, v^{\rho,\sigma}(x)= - c_k \dfrac{d}{dx}\left( p_{k-1} (v^{\rho+1, \,\sigma+1},x) \, v^{\rho+1, \, \sigma+1}(x) \, \right),
\end{equation}
holds true for $k\geq 1$. Consequently, a direct application of integration by parts to \eqref{eq:mom_log} yields
\begin{equation}\label{momento_k_log}
    \begin{split}
        M_k(y)&=\int_{-1}^1 \log\lvert x-y\rvert \, p_k (v^{\rho,\sigma},x) \, v^{\rho,\sigma}(x) \, dx =c_k\int_{-1}^1 \left(\dfrac{d\log \lvert x-y\rvert}{dx}\right) p_{k-1} (v^{\rho+1, \,\sigma+1},x) \, v^{\rho+1, \, \sigma+1}(x) \, dx \\
        & =c_k \int_{-1}^1 \dfrac{p_{k-1} (v^{\rho+1, \,\sigma+1},x)}{x-y} \, v^{\rho+1, \, \sigma+1}(x) \, dx. \qquad \qquad\qquad\qquad\qquad\qquad\qquad\qquad \ \ \ k\geq 1. 
    \end{split}
\end{equation}
In particular, since $\rho+\sigma+2\neq -1,$ we have for $k=1$ (see e.g. \cite[page 320]{gradshteyn2007table})
\begin{equation*}
     M_1(y)=c_1p_0(v^{\rho+1, \, \sigma+1}) \left( v^{\rho+1, \, \sigma+1}(y)\, \pi \cot(\pi \rho) - \frac{2^{\rho + \sigma+2} \Gamma(\rho+1) \Gamma(\sigma + 2)}{\Gamma(\rho + \sigma + 3)} \, {}_2F_1\left( -\rho - \sigma-2,\ 1;\ - \rho;\ \frac{1 - y}{2} \right)\right).
\end{equation*}
Moreover, for $k=2$
\begin{eqnarray*}M_2(y)&=&c_2 \frac{p_0(v^{\rho+1,\sigma+1})}{b_1(v^{\rho+1,\sigma+1})}\left[ \int_{-1}^1 v^{\rho+1,\sigma+1}(x)dx+(y-a_0(v^{\rho+1,\sigma+1}))\int_{-1}^1 \frac{v^{\rho+1,\sigma+1}(x)}{x-y}dx\right]\\
&=& c_2\left(2^{\rho+\sigma+4}\frac{\Gamma(\rho+2)\Gamma(\sigma+2)}{\Gamma(\rho+\sigma+4)}\right)^{-\frac{1}{2}}\sqrt{\dfrac{(\rho+\sigma+4)^2(\rho+\sigma+5)}{4(\rho+2)(\sigma+2)}}\left[\left(2^{\rho+\sigma+4}\frac{\Gamma(\rho+2)\Gamma(\sigma+2)}{\Gamma(\rho+\sigma+4)}\right)\right. \\
&+& \left. \left(y-\dfrac{(\sigma-\rho)(\sigma+\rho-2)}{(\rho+\sigma+4)(\rho+\sigma+6)}\right) \right]
\left(v^{\rho+1, \sigma+1}(y)\, \pi \cot(\pi \rho) - \frac{2^{\rho + \sigma+2} \Gamma(\rho+1) \Gamma(\sigma + 2)}{\Gamma(\rho + \sigma + 3)} \right. \times\\
            &&\left. {}_2F_1\left( -\rho - \sigma-2,\ 1;\ - \rho;\ \frac{1 - y}{2} \right)\right).
\end{eqnarray*}

Finally,  identity \eqref{eq:Ricorrenza_Momenti_Logaritmici} follows by combining the above expressions and performing straightforward algebraic manipulations. \\
To prove \eqref{eq:Momenti_Log_Particolari}, holding the assumptions in \eqref{eq:Cond_Pesi_Log}, 
by \eqref{momento_k_log} we have for $k\geq 1$ $M_k(y)=c_kQ_{k-1}(y)$ with
 \begin{equation*}
    Q_j(y)=\int_{-1}^1 p_{j} (v^{\rho+1, \,\sigma+1},x) \, (x-y)^{-1} \, v^{\rho+1, \, \sigma+1}(x) \, dx, \qquad \quad \text{for} \quad j\geq 0.
\end{equation*}
Then, the explicit expression \eqref{eq:Momenti_Log_Particolari} follows from \cite[Proposition~3.2]{HilbertII}.
\end{proof}

\section{Conclusions}\label{sec:concl}

In this paper, we introduced and analyzed the $VP$-numerical method, based on the de la Vall\'ee Poussin polynomial approximation, for the solution of second-kind Fredholm integral equations. The proposed approach was motivated by the need to overcome some limitations of  the classical $\mathcal{L}$-method, a projection method relying on Lagrange interpolation at Jacobi nodes. 

The theoretical analysis established the stability and convergence of the method in suitable weighted uniform spaces. In particular, the approximation error was shown to be governed by the error of the best weighted polynomial approximation, yielding convergence rates comparable with the optimal ones attainable in the considered functional setting. Unlike the $\mathcal{L}$-method, the proposed scheme benefits from uniformly bounded Lebesgue constants, which eliminate the logarithmic growth typically affecting Lagrange operators and provides a more robust approximation framework. Furthermore, the $VP$-method was shown to have a broader applicability, since its convergence can be guaranteed under assumptions that are less restrictive than those required by the corresponding Lagrange-based method, allowing the treatment of wider classes of weight functions and solution behaviors.

From a computational point of view, the $VP$-method leads to well-conditioned linear systems, whose condition numbers do not increase as the discretization dimension increases. This property is particularly relevant for large-scale computations, where numerical stability and reliability play a crucial role in the overall performance of the algorithm.

The numerical experiments fully confirmed the theoretical findings. Although the $VP$-method and the $\mathcal{L}$-method had comparable global maximum errors, the first consistently achieved superior local accuracy. In particular, the damping of Gibbs-type oscillations resulted in more accurate pointwise reconstructions of the solution, especially in critical regions characterized by steep variations or reduced smoothness. 

Overall, the proposed method combines favorable approximation properties, stability, wider applicability, and computational efficiency, making it a valuable alternative to classical interpolation-based techniques for Fredholm integral equations. 

\section*{Acknowledgments}

This work has been supported by the Gruppo Nazionale Calcolo Scientifico-Istituto Nazionale di Alta Matematica (GNCS-INdAM) and by the 2026 project ``Metodi numerici per modelli integrali e dinamiche con memoria'' .

This research has been accomplished within the RITA ``Research ITalian network on Approximation'', the UMI Group TAA `` Approximation Theory and Applications'', and the SIMAI Activity Group ANA$\&$A ``Numerical and Analytical Approximation of Data and Functions with Applications''. 

\subsection*{Memberships}

All the  authors are members of Gruppo Nazionale Calcolo Scientifico-Istituto Nazionale di Alta Matematica (GNCS-INdAM),  UMI-TAA \textquotedblleft Approximation Theory and Applications'' Research Group, RITA \textquotedblleft Research
 ITalian network on Approximation'' and SIMAI Activity Group ANA$\&$A \textquotedblleft Numerical and Analytical Approximation of Data and Functions with Applications''.

\section*{Conflict of interest}
The authors have no conflicts of interest to declare relevant to this article's content.

\bibliographystyle{plainurl}
\bibliography{bibliografia}

@book{kress2014,
  author    = {Kress, R.},
  title     = {Linear {I}ntegral {E}quations},
  edition   = {3rd},
  publisher = {Springer},
  year      = {2014},
  DOI = {10.1007/978-1-4614-9593-2}
}

@incollection{CapobiancoJLuthermastro,
title={Weighted uniform convergence
of the quadrature method for {Cauchy} singular integral equations},
author={Capobianco, M.R. and Junghanns, P.  and Luther, U.  and  Mastroianni, G.},
booktitle={Operators and Related Topics (Tel Aviv, 1995)},
Series={Oper. Theory Adv. Appl.},
volume={90},
pages={153-181},
year={1995},
publisher={Birkh{\"a}user Basel},
address={Basel},
DOI = {10.1007/978-3-0348-9040-3_5}
}

@article{dbmsiam,
title={Projection methods and condition numbers in uniform norm for {F}redholm and {C}auchy Integral equations},
author={De Bonis, M.C. and  Mastroianni, G.},
journal={SIAM  J. Numerical Analysis},
volume={44},
DOI = {10.1137/050626934},
pages={},
year={2006}
}

@article {mata,
    AUTHOR = {Occorsio, D.  and W. Themistoclakis},
    TITLE = {Some remarks on filtered polynomial interpolation at {C}hebyshev nodes},
    JOURNAL = {Dolomites Research Notes on Approximation},
    VOLUME = {14},
    issue={2},
    YEAR = {2021},
    pages={68-84},
    }

@article{Themistoclakis_1999, title={{Some Interpolating Operators of de la Vallée-Poussin Type}}, volume={84}, ISSN={1588-2632}, 
DOI={10.1023/a:1006637303487}, number={3}, journal={Acta Mathematica Hungarica}, publisher={Springer Science and Business Media LLC}, author={Themistoclakis, W.}, year={1999}, month=Aug, pages={221–235} }

@article{De_Bonis_2021, title={Filtered interpolation for solving Prandtl’s integro-differential equations}, volume={88}, ISSN={1572-9265}, 
DOI={10.1007/s11075-020-01053-x}, number={2}, journal={Numerical Algorithms}, publisher={Springer Science and Business Media LLC}, author={De Bonis, M. C. and Occorsio, D. and Themistoclakis, W.}, year={2021}, month=Feb, pages={679–709} }

@article{Airfoil,
title = {A numerical method for the generalized airfoil equation based on the de la {Vallée Poussin} interpolation},
journal = {Journal of Computational and Applied Mathematics},
volume = {180},
number = {1},
pages = {71-105},
year = {2005},
issn = {0377-0427},
doi = {https://doi.org/10.1016/j.cam.2004.10.003},
author = {G. Mastroianni and W. Themistoclakis},
}

@article {capothemi,
    AUTHOR = {Capobianco, M.R. and Themistoclakis, W.},
    TITLE = {On the boundedness of de la {V}all\`ee {P}oussin operators},
    JOURNAL = {East J. Approx,},
    VOLUME = {},
    YEAR = {2001},
    pages={},
    }

@book{mastromilobook,
    AUTHOR = {Mastroianni,  G. and  Milovanovi\'c, G. V.},
     TITLE = {Interpolation Processes Basic Theory and Applications},
 PUBLISHER = {Springer Verlag},
   ADDRESS = {Berlin},
   SERIES = {Springer Monographs in Mathematics},
    VOLUME = {},
      YEAR = {2009},
      DOI = {10.1007/978-3-540-68349-0}
}

@book{JMN,
    AUTHOR = {Junghanns, P. and  Mastroianni,  G. and  Notarangelo, I.},
     TITLE = {Weighted {P}olynomial {A}pproximation and {N}umerical {M}ethods for {I}ntegral {E}quations},
 PUBLISHER = {Birkh\"auser, Cham },
   ADDRESS = {Basel, Switzerland},
   SERIES = {Pathways in Mathematics},
    VOLUME = {},
    YEAR = {2021},
    DOI = {10.1016/S0377-0427(96)00128-8}
}

@ARTICLE{Occothemi,
	author = {Occorsio, D. and Themistoclakis, W.},
	title = {On the filtered polynomial interpolation at {C}hebyshev nodes},
	year = {2021},
	journal = {Appl. Numer. Math.},
	volume = {166},
	pages = {272--287},
    doi = {10.1016/j.apnum.2021.04.013},
	}

@ARTICLE{themi2011,
	author = {Themistoclakis, W.},
	title = {Uniform approximation on $[-1, 1]$ via discrete de la {V}all\'ee {P}oussin means},
	year = {2012},
    DOI={10.1007/s11075-012-9588-4},
	journal = {Numer. Algorithms},
	volume = {60},
	pages = {593--612},
}

@ARTICLE{themiL1,
	author = {Themistoclakis, W.},
	title = {Weighted $L^1$ approximation on $[-1, 1]$ via discrete de la {V}all\'ee {P}oussin means},
	year = {2018},
	journal = {Math. Comput. Simul.},
    DOI={10.1016/j.matcom.2017.06.005},
	volume = {147},
	pages = {279--292},
}

@ARTICLE{themiBarel,
	author = {Themistoclakis, W. and Van Barel, M.},
	title = {Generalized de la {V}all\'ee {P}oussin approximations on $[-1, 1]$},
	year = {2017},
	journal = {Numer. Algorithms},
    DOI = {10.1007/s11075-016-0194-8},
	volume = {75},
	pages = {1--31},
}

@article{ORT_2024,
author = {Occorsio, D. and Russo, M.G. and Themistoclakis, W.},
title = {On solving some {Cauchy} singular integral equations by de la {Vallée Poussin} filtered approximation},
journal = {Appl. Num. Math.},
volume = {200},
pages = {358--378},
year = {2024},
note = {New Trends in Approximation Methods and Numerical Analysis (FAATNA20>22)},
doi = {https://doi.org/10.1016/j.apnum.2023.07.022},
}

@article{OT_Dolomites,
    author = {Occorsio, D. and Themistoclakis, W.},
    title = {Some remarks on filtered polynomial interpolation at {C}hebyshev nodes},
    journal = {Dolomites Res. Notes Approx},
    year = {2021},
    volume = {14},
    issue = {2},
    pages = {68--84},
    doi = {10.14658/PUPJ-DRNA-2021-2-9}
}

@book{DT,
	author = {Ditzian, Z. and Totik, W.},
	title = {Moduli of smoothness},
	publisher = {SCMG Springer-Verlag},
	address = {New York Berlin Heidelberg London Paris Tokyo},
	year = {1987},
}

@book{Szego,
  title={Orthogonal Polynomials},
  author={Szeg{\"o}, G.},
  number={v. 23},
  isbn={9780821889527},
  lccn={a61000607},
  series={American Mathematical Society colloquium publications},
  year={1959},
  publisher={American Mathematical Society}
}

@book{atkinson,
	author = {Atkinson, K. E.},
	title = {The Numerical Solution of Integral Equations of the Second Kind },
	publisher = {Cambridge Monographs on Applied and Computational Mathematics, Series Number 4},
	address = {},
	year = {1997},
}

@article{JL,
author={Junghanns, P.  and Luther, U. },
title= {Cauchy singular integral equations in spaces of continuous functions and methods for their numerical solution.},
journal={J. Comp. Appl. Math},
volume    = {77},
pages     = {201--237},
year  ={1997},
DOI = {10.1016/S0377-0427(96)00128-8}
}

@article{HilbertII,
  author    = {D. Occorsio and M. G. Russo and W. Themistoclakis},
  title     = {Filtered integration rules for finite weighted {Hilbert} transforms {II}},
  journal   = {Dolomites Research Notes on Approximation},
  volume    = {15},
  number    = {3},
  pages     = {93--104},
  year      = {2022},
  DOI = {10.14658/PUPJ-DRNA-2022-3-9}
}

@book{gradshteyn2007table,
  title= {Table of Integrals, Series, and Products},
  author= {Gradshteyn, I. S. and Ryzhik, I. M.},
  edition   = {7},
  publisher = {Academic Press},
  year      = {2007},
  address   = {Amsterdam},
  note      = {Translated from Russian by Scripta Technica, Inc.}
}

@book{VP1919,
  author    = {de la Vall{\'e}e Poussin, C.-J.},
  title     = {Le{\c c}ons sur l'Approximation des {F}onctions d'une Variable R{\'e}elle},
  publisher = {Gauthier-Villars},
  address   = {Paris},
  year      = {1919},
  note       = {Reprinted by Chelsea Publishing Company, New York, 1970}
}

@InProceedings{iwota,
    author={Mastroianni, G. and Russo, M.G. and Themistoclakis, W.},
    title={Numerical {M}ethods for {C}auchy {S}ingular {I}ntegral {E}quations in {S}paces of {W}eighted {C}ontinuous {F}unctions},
    booktitle={Recent Advances in Operator Theory and its Applications},
    year={2005},
    publisher={Birkh{\"a}user Basel},
    address={Basel},
    pages={311--336},
    doi={10.1007/3-7643-7398-9_15}
}

@article{KangMom,
	author = {Kang, H. and Xiang, S. and He, G.},
	title = {Computation of integrals with oscillatory and singular integrands using {C}hebyshev expansions},
	journal = {J. Comput. Appl. Math.},
	volume = {242},
	pages = {141--156},
	year = {2013},
	doi = {https://doi.org/10.1016/j.cam.2012.10.016}
}

@article{PiessBraOsc,
	author = {Piessens, R. and Branders, M.},
	title = {On the computation of {F}ourier transforms of singular functions},
	journal = {J. Comput. Appl. Math.},
	volume = {43},
	number = {1},
	pages = {159--169},
	year = {1992},
	doi = {https://doi.org/10.1016/0377-0427(92)90264-X}
}

\end{document}